\documentclass[dvipdfmx]{amsart}
\usepackage{amsmath,amssymb,bm}
\usepackage[dvipdfmx]{graphicx}
\usepackage{color}
\usepackage{here}
\newtheorem{thm}{Theorem}[section]
\newtheorem{lem}[thm]{Lemma}

{\theoremstyle{definition} 
\newtheorem{rem}[thm]{Remark}
\newtheorem{ex}[thm]{Example}}

\newtheorem{prop}[thm]{Proposition}

\begin{document}
%
%
%
%
%
%
%
\title[]
{On annulus presentations, dualizable patterns and RGB-diagrams}
\author{Keiji Tagami}
\address{The Faculty of Economic Sciences, Hiroshima Shudo University, Hiroshima 731-3195, JAPAN
}
\email{ktagami@shudo-u.ac.jp}
\date{\today}
\maketitle
%
\begin{abstract}
The $0$-trace of a knot is the $4$-manifold represented by the $0$-framing of the knot. 
In this manuscript, we survey methods constructing a pair of knots with diffeomorphic $0$-traces. 
In particular, we focus on Gompf-Miyazaki's dualizable pattern, Abe-Jong-Omae-Takeuchi's band presentation, and RGB-diagram given by Piccirillo and named by the author, and we draw the relations among these methods directly. 
As an application, we give a sufficient condition that two knots obtained by  Abe-Jong-Omae-Takeuchi's method coincide. 
\end{abstract}

\section{Introduction}
A knot in $\mathbf{S}^3=\partial \mathbf{B}^4$ is {\it smoothly slice} if it bounds a proper and smooth disk in $\mathbf{B}^4$. 
We can find many motivations to study the smooth sliceness of knots, for example:
\begin{itemize}
\item If the smooth $4$-dimensional Poincar{\' e} conjecture is true, then for two knots $K$ and $K'$ having homeomorphic $0$-surgeries $K$ is smoothly slice if and only if $K'$ is smoothly slice (\cite[Lemma~3.2]{AJOT}). 
\item It is conjectured that any smoothly slice knot is a ribbon knot (Slice-Ribbon Conjecture) \cite{Fox}. 
\end{itemize}
\par 
The {\it $0$-trace} of a knot is the $4$-manifold obtained from $\mathbf{B}^4$ by attaching a $2$-handle along the $0$-framing of the knot. 
The following theorem (Theorem~\ref{thm:slice-embed}) implies that the $0$-trace of a knot has the complete information to determine whether the knot is smoothly slice or not. 
%
\footnote{
It seems that Theorem~\ref{thm:slice-embed} has been known to the experts. 
However, the author cannot find the proper reference. 
For example, we can find a kind proof for this theorem in \cite[Theorem~1.8]{Miller-Piccirillo}. 
}
\begin{thm}\label{thm:slice-embed}
A knot is smoothly slice if and only if its $0$-trace smoothly embeds in $\mathbf{S}^4$. 
\end{thm}
%
\par 
Akbulut and Kirby \cite[Problem~1.19]{Kirby2} conjectured that two knots having homeomorphic $0$-surgeries are concordant. 
As mentioned above, if one of the two knots is smoothly slice, this conjecture is true under the smooth $4$-dimensional Poincar{\'e} conjecture. 
However, it has been proved that this conjecture is false. 
In fact, Yasui \cite{Yasui} gave infinitely many counterexamples for Akbulut-Kirby's conjecture. 
\par 
Gompf and Miyazaki \cite[Proposition~3.1]{Gompf-Miyazaki} gave a pair of knots which have homeomorphic $0$-surgeries and whose connected sum is not ribbon by utilizing a pattern, which is called a {\it dualizable pattern} in this manuscript (for definition, see Section~\ref{sec:dualizable}). 
In particular, there is no ribbon concordance between them in both directions. 
Abe and the author \cite{Abe-Tagami} also gave such a pair of knots. 
In \cite{Abe-Tagami}, we use the technique called ``annulus twist" and ``annulus presentation", which are essentially given by Osoinach \cite{Osoinach} and improved in \cite{Teragaito, AJOT}. 
By utilizing \cite[Theorem~2.8]{AJOT}, we see that Abe and the author's knots have diffeomorphic $0$-traces. 
We also see that Gompf and Miyazaki's knots have diffeomorphic $0$-traces by \cite[Theorem~3.1]{Miller-Piccirillo}. 
Hence, we can consider the following question: 
\begin{center}
Are two knots having diffeomorphic $0$-traces concordant? 
\end{center}
\par
This question has been negatively solved by Miller and Piccirillo \cite{Miller-Piccirillo}. 
In fact, they gave infinitely many pairs of knots such that they have diffeomorphic $0$-traces and yet are distinct in smooth concordance by using dualizable patterns. 
They also mentioned a relation between annulus presentations and dualizable patterns. 
\par 
Recently, Piccirillo \cite{Piccirillo} introduced a class of Kirby diagrams $R\cup G\cup B$. 
In this manuscript, we call a Kirby diagram of the class an {\it RGB-diagram} (for definition, see Section~\ref{sec:RGB}). 
As an application, Piccirillo \cite{Piccirillo2} constructed a non-slice knot whose $0$-trace is diffeomorphic to that of the Conway knot. 
In particular, we see that the Conway knot is not smoothly slice by Theorem~\ref{thm:slice-embed}. 
We remark that Piccirillo's construction can be explained in terms of annulus presentations and annulus twists (see Remark~\ref{rem:unknotting-one}). 
\par 
In this manuscript, we survey annulus presentations, dualizable patterns and RGB-diagrams, and we draw an RGB-diagram from an annulus presentation explicitly. 
In particular, we clarify the relation among ``special" annulus presentations (defined in Section~\ref{sec:annulus}), dualizable patterns and RGB-diagrams (see Theorems~\ref{thm:annulus-RGB} and Figure~\ref{figure:relation}). 
Moreover, we extend this relation to the oriented case (Theorem~\ref{thm:annulus-RGB-oriented}). 
As an application, we give a sufficient condition for a knot with an annulus presentation to be preserved under the corresponding annulus twist (Theorem~\ref{thm:trivial}). 
\begin{figure}[h]
\centering
\includegraphics[scale=1.3]{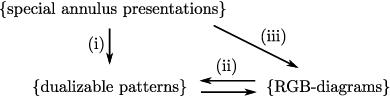}
\caption{Relations among special annulus presentations, dualizable patterns and RGB-diagrams. 
(i) is given by Miller-Piccirillo \cite[Section~5]{Miller-Piccirillo}. 
(ii) and its inverse are given by Piccirillo \cite[Proposition~4.2]{Piccirillo}. 
(iii) is given in Theorem~\ref{thm:annulus-RGB}.
The inverses of (i) and (iii) do not exist (Remark~\ref{rem:inverse}). 
These relations can be extended to oriented versions (see Sections~\ref{sec:annulus-dualizable-oriented} and \ref{sec:RGB-diagram-oriented}, and Theorem~\ref{thm:annulus-RGB-oriented}). }
\label{figure:relation}
\end{figure}

%
%

\subsection{Notation}
Throughout this manuscript, 
\begin{itemize}
\item we denote the $3$-manifold obtained from $\mathbf{S}^3$ by applying $n$-framed surgery on a knot $K\subset \mathbf{S}^3$ by $M_{K}(n)$, 
\item the {\it $n$-trace} of a knot $K$ is the $4$-manifold obtained from $\mathbf{B}^4$ by attaching a $2$-handle along an $n$-framed knot $K\subset \mathbf{S}^3$ and we denote it by $X_{K}(n)$, 
\item we denote an open regular neighborhood of a submanifold $P$ in a manifold $V$ by $\nu(P)$, 
\item we denote the unknot in $\mathbf{S}^3$ by $U$, 
\item unless specifically mentioned, all knots and links are smooth and unoriented, and all other manifolds are smooth and oriented, 
\item for a manifold $M$, define $-M$ to be the manifold obtained by reversing the orientation, 
\item for a Kirby diagram $L$, we denote the $4$-manifold that $L$ represents by $X_{L}$, and
\item we will use $\cong $ to denote orientation-preservingly differomorphic $4$-manifolds or homeomorphic $3$-manifolds. 
\end{itemize}
%
%
\section{Annulus presentation}\label{sec:annulus}
%
\subsection{Annulus twist}
Let $A\subset \mathbf{S}^{3}$ be an embedded annulus with $\partial A=c_{1}\cup c_{2}$. 
%
An {\it $n$-fold annulus twist along $A$} is to apply $(\operatorname{lk}(c_1,c_2)+1/n)$-surgery on $c_{1}$ and $(\operatorname{lk}(c_1,c_2)-1/n)$-surgery on $c_{2}$, where $\operatorname{lk}(c_1,c_2)$ is the linking number of $c_1$ and $c_2$, and we give $c_1$ and $c_2$ parallel orientations. 
We see that the resulting manifold obtained by an annulus twist is $\mathbf{S}^{3}$. 
%
%
%
%
%
%
\subsection{Annulus presentation}
Let $A\subset \mathbf{S}^{3}$ be an embedded annulus with $\partial A=c_{1}\cup c_{2}$. 
Take an embedding of a band $b\colon I\times I\rightarrow \mathbf{S}^{3}$ such that 
\begin{itemize}
\item $b(I\times I)\cap \partial A=b(\partial I\times I)$, 
\item $b(I\times I)\cap \operatorname{Int} A$ consists of ribbon singularities, and 
\item $A\cup b(I\times I)$ is an immersion of an orientable surface, 
\end{itemize}
where $I=[0, 1]$. 
If a knot $K\subset \mathbf{S}^3$ is isotopic to the knot $(\partial A\setminus b(\partial I\times I))\cup b(I\times \partial I)$, 
then we call $(A, b)$ an {\it annulus presentation} of $K$. 
An annulus presentation $(A,b)$ is {\it special} if $A$ is a Hopf band. 
%


%
\begin{rem}\label{rem:bandpre}
In the definition of special annulus presentations, if we omit the condition that $A\cup b(I\times I)$ is an immersion of an orientable surface, it coincides with the definition of band presentations defined by Abe, Jong, Omae and Takeuchi \cite{AJOT}.  
In particular, a special annulus presentation is a band presentation. 
Note that in \cite{AJLO,abe-tange}, our special annulus presentations are called ``annulus presentations". 
\end{rem}
\begin{ex}
The knot $6_{3}$ has a special annulus presentation $(A, b)$ (Figure~\ref{figure:annulus-pre}). 
\end{ex}
\begin{figure}[h]
\centering
\includegraphics[scale=0.76]{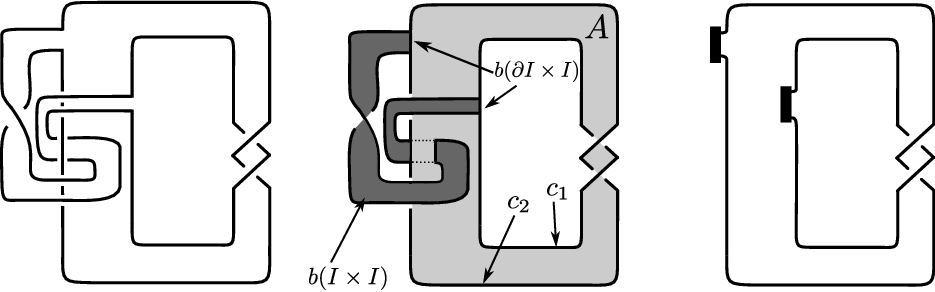}
\caption{A special annulus presentation of $6_3$. For simplicity, we often draw an annulus presentation as the right picture, where the attaching regions for the band are represented by bold arcs and the bands are omitted. }
\label{figure:annulus-pre}
\end{figure}
Let $K$ be a knot with an annulus presentation $(A, b)$. 
Let $\widetilde{A}\subset A$ be a shrunken annulus with $\partial\widetilde{A}=\widetilde{c}_{1}\cup \widetilde{c}_{2}$ which satisfies the following: 
\begin{itemize}
\item $\overline{A\setminus \widetilde{A}}$ is a disjoint union of two annuli, 
\item each $\widetilde{c}_{i}$ is isotopic to $c_{i}$ in $\overline{A\setminus \widetilde{A}}$ for $i=1,2$, and 
\item $A\setminus (\partial A\cup \widetilde{A})$ does not intersect $b(I\times I)$. 
\end{itemize}

Then, by $A^{n}(K)$, 
we denote the knot obtained from $K$ by  the $n$-fold annulus twist along $\widetilde{A}$. 
For simplicity, we also use $A(K)$ instead of $A^{1}(K)$. 
%
\begin{ex}
We consider the knot $6_{3}$
with the annulus presentation $(A,b)$ in  Figure \ref{figure:annulus-pre}.
Then $A(6_{3})$ is the right picture in Figure \ref{figure:annulus-twist}.
\end{ex}
\begin{figure}[h]
\centering
\includegraphics[scale=0.75]{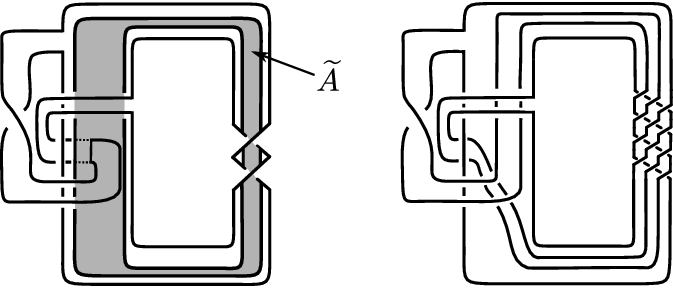}
\caption{A shrunken annulus $\tilde{A}$ for an annulus presentation of $6_{3}$ (left) and $A(6_{3})$ (right). 
}\label{figure:annulus-twist}
\end{figure}
The following theorem is essentially due to Osoinach \cite[Theorem 2.3]{Osoinach}. 
%
%
%
\begin{thm}\label{thm:Osoinach} 
Let $K\subset \mathbf{S}^3$ be a knot with an annulus presentation $(A, b)$. 
Then, for any $n\in\mathbf{Z}$ there is an orientation-preservingly homeomorphism $\phi_{n}\colon M_{K}(0) \rightarrow M_{A^{n}(K)}(0)$. 
\end{thm}
The homeomorphism $\phi_{n}$ in Theorem~\ref{thm:Osoinach} is concretely given by Figure~\ref{figure:Osoinach-homeo} (see also \cite{Teragaito}). 
\par 
\begin{figure}[h]
\centering
\includegraphics[scale=0.75]{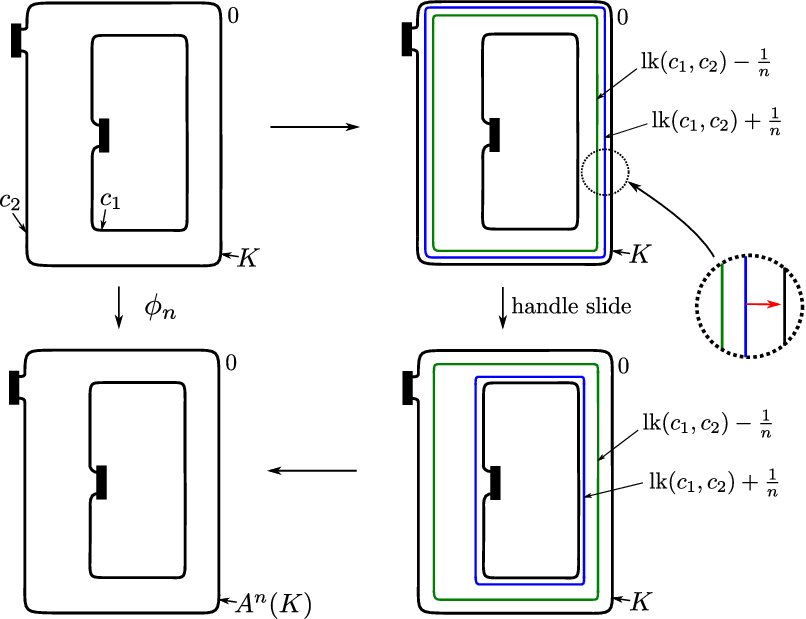}
\caption{(color online) Homeomorphism $\phi_{n}$. In this picture the annulus $A$ may be knotted and twisted. 
} 
\label{figure:Osoinach-homeo}
\end{figure}
Abe, Jong, Omae and Takeuchi \cite[Theorem~2.8]{AJOT} partially extended Theorem~\ref{thm:Osoinach}. 
In particular, by applying their result to a knot with a special annulus presentation, we obtain the following. 
\begin{thm}\label{thm:AJOT} 
Let $K\subset \mathbf{S}^3$ be a knot with a special annulus presentation $(A, b)$. 
Then, the homeomorphism $\phi_{n}\colon M_{K}(0) \rightarrow M_{A^{n}(K)}(0)$ given by Figure~\ref{figure:Osoinach-homeo} extends to an orientation-preservingly diffeomorphism $\Phi_{n}\colon X_{K}(0)\rightarrow X_{A^{n}(K)}$ for any $n\in\mathbf{Z}$. 
\end{thm}
%
%
%
%
\section{Dualizable pattern}\label{sec:dualizable}
In this section, we recall the definition of dualizable patterns. 
For details on dualizable patterns, see \cite{Gompf-Miyazaki} and \cite{Miller-Piccirillo}. 
%
\subsection{Dualizable pattern}
Let $P\colon \mathbf{S}^1\rightarrow V$ be an oriented knot in a solid torus $V=\mathbf{S}^1\times D^2$. 
Such a $P$ is called a {\it pattern}. 
By an abuse of notation, we use the notation $P$ for both a map and its image. 
Suppose that the image $P(\mathbf{S}^1)$ is not null-homologous in $V$. 
Define $\lambda_{V}$, $\mu_{P}$, $\mu_{V}$ and $\lambda_{P}$ as follows: 
\begin{itemize}
\item put $\lambda_{V}=\mathbf{S}^1\times \{x_0\}\subset \partial V \subset V$ for some $x_0\in \partial D^{2}$ and orient $\lambda_{V}$ so that $P$ is homologous to $k\lambda_{V}$ in $V$ for some positive $k\in\mathbf{Z}_{>0}$, 
\item define $\mu_{P}\subset V$ by a meridian of $P$ and orient $\mu_{P}$ so that the linking number of $P$ and $\mu_{P}$ is $1$, 
\item put $\mu_{V}=\{x_1\}\times \partial D^2\subset \partial V \subset V$ for some $x_{1}\in \mathbf{S}^1$ and orient $\mu_{V}$ so that $\mu_{V}$ is homologous to $l\mu_{P}$ in $V\setminus \nu(P)$ for some positive $l\in\mathbf{Z}_{>0}$, and 
\item define $\lambda_{P}$ by a longitude of $P$ which is homologous to $m\lambda_{V}$ in $V\setminus \nu(P)$  for some positive $m\in\mathbf{Z}_{>0}$. 
\end{itemize}
\par
For an oriented knot $K\subset \mathbf{S}^3$, let $\iota_{K}\colon V\rightarrow \mathbf{S}^3$ be an embedding which identifies $V$ with $\overline{\nu(K)}$ and 
sends $\lambda_{V}$ to the $0$-framing of $K$. 
%
%
Then $\iota_{K}\circ P\colon \mathbf{S}^1\rightarrow \mathbf{S}^3$ represents an oriented knot. 
The knot is called the {\it satellite} of $K$ with pattern $P$ and denoted by $P(K)$. 
\par
A pattern $P\colon \mathbf{S}^{1}\rightarrow V$ is {\it dualizable} if there is a pattern $P^{\ast}\colon \mathbf{S}^{1}\rightarrow V^{\ast}$ and an orientation-preserving homeomorphism $f\colon V\setminus \nu(P)\rightarrow V^{\ast}\setminus \nu(P^{\ast})$ such that 
$f(\lambda_{V})=\lambda_{P^{\ast}}$, $f(\lambda_{P})=\lambda_{V^{\ast}}$, $f(\mu_{V})=-\mu_{P^{\ast}}$ and $f(\mu_{P})=-\mu_{V^{\ast}}$. 
We call $P^{\ast}$ the {\it dual} of $P$. 
\par
Miller and Piccirillo \cite[Proposition~2.5]{Miller-Piccirillo} introduced a convenient technique to determine whether a given pattern is dualizable as follows (see also \cite[Section~2]{Gompf-Miyazaki}). 
Define $\Gamma\colon \mathbf{S}^1\times D^{2} \rightarrow \mathbf{S}^1\times \mathbf{S}^2$ by $\Gamma(t,d)=(t,\gamma(d))$, where $\gamma\colon D^{2}\rightarrow \mathbf{S}^2$ is an arbitrary orientation-preserving embedding. 
For any curve $c\colon \mathbf{S}^1\rightarrow \mathbf{S}^1\times D^{2}$, define $\widehat{c}=\Gamma\circ c \colon \mathbf{S}^1\rightarrow \mathbf{S}^1\times \mathbf{S}^2$. 
Then, we obtain the following proposition. 
\begin{prop}[{\cite[Proposition~2.5]{Miller-Piccirillo}}]\label{prop:MP}
A pattern $P$ in a solid torus $V$ is dualizable if and only if $\widehat{P}$ is isotopic to $\widehat{\lambda_{V}}$ in $\mathbf{S}^{1}\times \mathbf{S}^2$. 
\end{prop}
\begin{proof}
For the sake of completeness, we introduce the proof due to Miller and Piccirillo. 
Suppose that $\widehat{P}$ is isotopic to $\widehat{\lambda_{V}}$ in $\mathbf{S}^{1}\times \mathbf{S}^2$. 
Then, $V^{\ast}=\mathbf{S}^{1}\times \mathbf{S}^2\setminus \nu(\widehat{P})\cong \mathbf{S}^{1}\times \mathbf{S}^2\setminus \nu(\widehat{\lambda_{V}})$ is a solid torus. 
We identify $V^{\ast}$ with $\mathbf{S}^{1}\times D^2$ so that  $\widehat{\lambda_{P}}\subset V^{\ast}$ is identified with $\lambda_{V^{\ast}}=\mathbf{S}^{1}\times \{x_{0}\}$ for some $x_0\in D^{2}$. 
Define a pattern $Q$ by $\widehat{\lambda_{V}}\subset V^{\ast}$. 
Note that the orientation of $\widehat{\lambda_{V}}$ is determined by that of $\widehat{P}$. 
Since $\mathbf{S}^{1}\times \mathbf{S}^2\setminus \nu(\widehat{\lambda_{V}})\cong V$, we have 
\begin{align*}
V\setminus \nu(P)
\cong \mathbf{S}^{1}\times \mathbf{S}^2\setminus \nu(\widehat{P}\cup \widehat{Q})
\cong V^{\ast}\setminus \nu(Q). 
\end{align*}
This induces an orientation-preserving homeomorphism $f\colon V\setminus \nu(P)\rightarrow V^{\ast}\setminus \nu(Q)$ such that 
$f(\lambda_{V})=\lambda_{Q}$, $f(\lambda_{P})=\lambda_{V^{\ast}}$, $f(\mu_{V})=-\mu_{Q}$ and $f(\mu_{P})=-\mu_{V^{\ast}}$. 
Hence, we have $P^{\ast}=Q$. 
\par
Conversely, suppose that a pattern $P$ in a solid torus $V$ is dualizable. 
Then, there is an orientation-preserving homeomorphism $f\colon V\setminus \nu(P)\rightarrow V^{\ast}\setminus \nu(P^{\ast})$ such that 
$f(\lambda_{V})=\lambda_{P^{\ast}}$, $f(\lambda_{P})=\lambda_{V^{\ast}}$, $f(\mu_{V})=-\mu_{P^{\ast}}$ and $f(\mu_{P})=-\mu_{V^{\ast}}$. 
Let $V_1$ be a solid torus and $\mu_{V_{1}}$ be a meridian of $V_{1}$. 
Then, we have 
\begin{align*}
\mathbf{S}^{1}\times \mathbf{S}^2\setminus \nu(\widehat{P})
&\cong (V\setminus \nu(P))\cup_{\partial} V_{1}\\
&\cong (f(V\setminus \nu(P)))\cup_{\partial} V_{1}\\
&= (V^{\ast}\setminus \nu(P^{\ast}))\cup_{\partial} V_{1}\\
&=V^{\ast}, 
\end{align*}
where the first union is given by identifying $\mu_{V}$ with $\mu_{V_{1}}$, and the second and the third unions are given by identifying $-\mu_{P^{\ast}}$ with $\mu_{V_{1}}$.  
Hence, $\widehat{P}$ is a knot in $\mathbf{S}^{1}\times \mathbf{S}^2$ whose complement is a solid torus. 
By Waldhausen \cite{Waldhausen}, it is proved that all genus one Heegaard splittings of $\mathbf{S}^{1} \times \mathbf{S}^{2}$ are isotopic. 
Hence, $\widehat{P}$ is isotopic to $\pm \widehat{\lambda_{V}}$. 
Here, by the definition of $\lambda_{V}$, $P$ is homologous to a $m\lambda_{V}$ in $V\setminus \nu(P)$  for some positive $m\in\mathbf{Z}_{>0}$, and we see that $\widehat{P}$ is isotopic to $+\widehat{\lambda_{V}}$. 
\end{proof}
%
By Proposition~\ref{prop:MP}, we can draw the dual $P^{\ast}$ for a given dualizable pattern $P$ as in Figure~\ref{figure:dualizable} (see also \cite[Section~4]{Gompf-Miyazaki} and \cite[Figure~2]{Miller-Piccirillo}). 
%
%
\begin{figure}[h]
\centering
\includegraphics[scale=0.86]{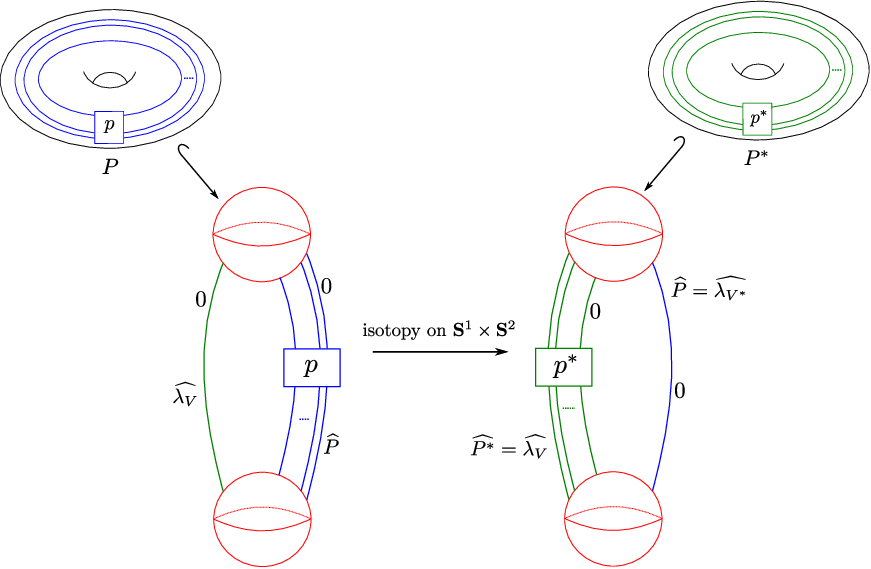}
\caption{(color online) A dualizable pattern $P$ and its dual $P^{\ast}$. 
The pairs of two balls represent $\mathbf{S}^{1}\times \mathbf{S}^{2}$. 
To clarify the parameters of the tori $V=\mathbf{S}^{1}\times \mathbf{S}^{2}\setminus \nu(\widehat{\lambda_{V}})$ and $V^{\ast}=\mathbf{S}^{1}\times \mathbf{S}^{2}\setminus \nu(\widehat{P})$, 
we draw the $0$-framings of $\widehat{\lambda_{V}}$ and $\widehat{\lambda_{V^{\ast}}}$. }
\label{figure:dualizable}
\end{figure}
\begin{rem}
Miller and Piccirillo \cite{Miller-Piccirillo} commented that Waldhausen \cite{Waldhausen} only proved the uniqueness of the genus one Heegaard splitting of $\mathbf{S}^{1} \times \mathbf{S}^{2}$ up to differmorphism and we require more work to prove the uniqueness up to isotopy. 
For a complete proof, for example, we use the following two facts (see also \cite{Carvalho-Oertel}). 
\begin{itemize}
\item For any orientable surface $F$, every Heegaard splitting of $F\times \mathbf{S}^{1}$ is a stabilization of the standard one \cite{Schultens}. 
\item Two stabilizations with the same genus of the same Heegaard splitting are isotopic (for example, see \cite{Johnson}). 
Namely, a stabilization of a Heegaard splitting is uniquely defined. 
\end{itemize}
\end{rem}
\begin{rem}
Baker and Motegi \cite{Baker-Motegi} gave another description of dualizable pattern as follows. 
Let $k\cup c$ be a two-component link in $\mathbf{S}^3$ such that $c$ is the unknot and that the $(0,0)$-surgery on $k\cup c$ results in $\mathbf{S}^3$. 
In the proof of \cite[Lemma~2.3]{Baker-Motegi}, Baker and Motegi proved that $k$ is isotopic to an $\mathbf{S}^{1}$-fiber in the standard product structure of $M_{c}(0)\cong \mathbf{S}^1\times \mathbf{S}^2$ by utilizing Gabai's work \cite[Corollary~8.3]{Gabai}. 
In particular, $k$ is isotopic to a meridian of $c$ in $M_{c}(0)$. 
By Proposition~\ref{prop:MP}, we see that $k\subset \mathbf{S}^3\setminus \nu(c)\cong \mathbf{S}^1\times D^{2}$ represents a dualizable pattern. 
Moreover, if we give the solid torus $\mathbf{S}^3\setminus \nu(c)\cong \mathbf{S}^1\times D^{2}$ a parameter in a standard way so that $k(U)=k$, then for the dual $k^{\ast}$, we see that $k^{\ast}(U)$ is the surgery dual to $c$ in the surgered $\mathbf{S}^3$. 
Conversely, let $P\colon \mathbf{S}^1\rightarrow V$ be a dualizable pattern and $\mu_{V}$ be a meridian of $V$. 
We regard $V$ as a standard solid torus in $\mathbf{S}^3$. 
Then the two-component link $P\cup \mu_{V}$ in $\mathbf{S}^3$ satisfies $M_{P\cup \mu(V)}(0,0)\cong \mathbf{S}^3$ since $\widehat{P}$ is isotopic to $\widehat{\lambda_{V}} $ in $M_{\mu_{V}}(0)\cong \mathbf{S}^1\times \mathbf{S}^2$. 
\end{rem}

\begin{rem}\label{rem:winding}
Let $P$ be a pattern. 
Then, $P$ is homologous to $m\lambda_{V}$ for some non-negative $m\in\mathbf{Z}$ in $V$. 
We call $m$ the {\it algebraic winding number} of $P$. 
The {\it geometric winding number} of $P$ is the minimal number of intersections between a meridian disk of $V$ and a pattern which is isotopic to $P$ in $V$. 
By Proposition~\ref{prop:MP}, we see that
\begin{itemize}
\item any geometric winding number one pattern is dualizable and its dual is itself (see Figure~\ref{figure:gwinding}, and see also \cite[Figure~2]{Gompf-Miyazaki} and \cite[Lemma~2.4]{Baker-Motegi}). 
\item the algebraic winding number of any dualizable pattern is one (see Figure~\ref{figure:awinding}). 
\end{itemize}
\end{rem}
\begin{figure}[h]
\centering
\includegraphics[scale=0.37]{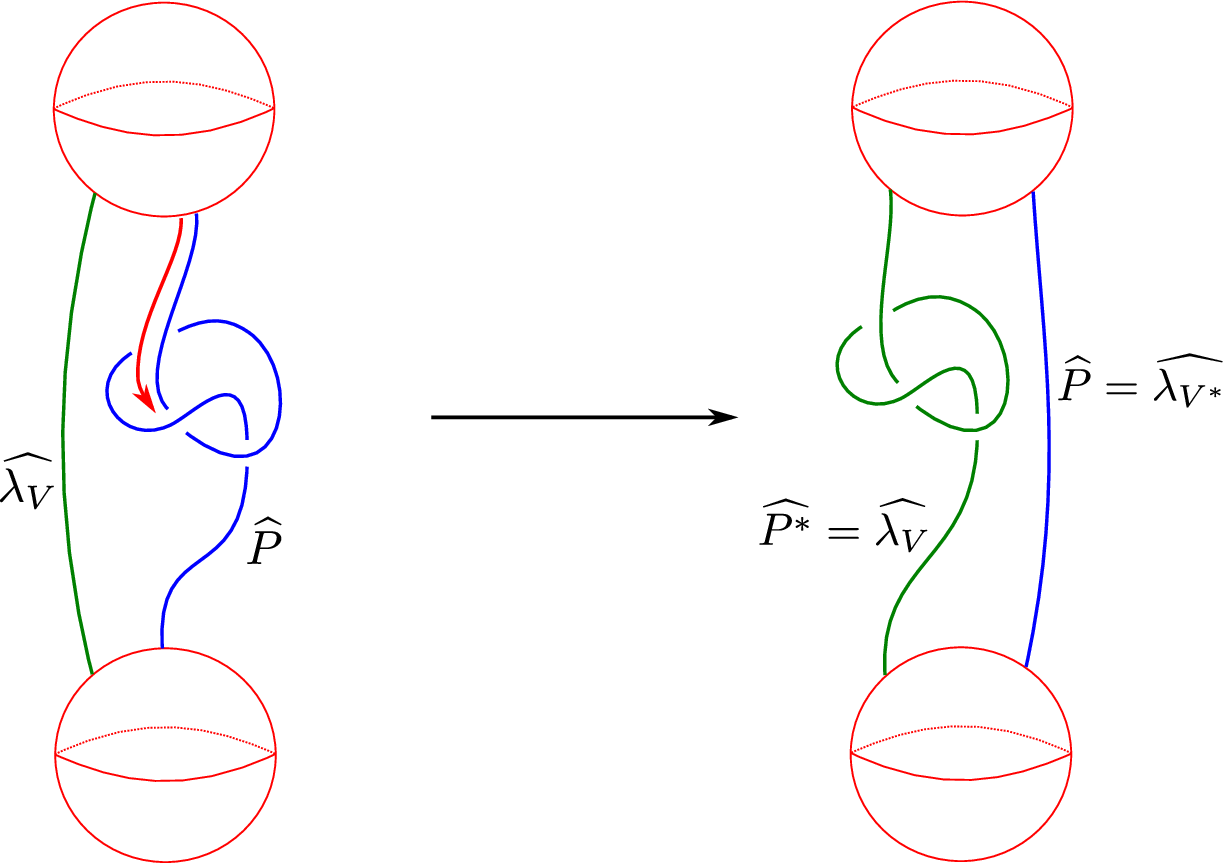}
\caption{(color online) 
The left picture is $\widehat{P}$ for a geometric winding number one pattern $P$. 
The pairs of two balls represent $\mathbf{S}^{1}\times \mathbf{S}^{2}$. 
}
\label{figure:gwinding}
\end{figure}
\begin{figure}[h]
\centering
\includegraphics[scale=0.5]{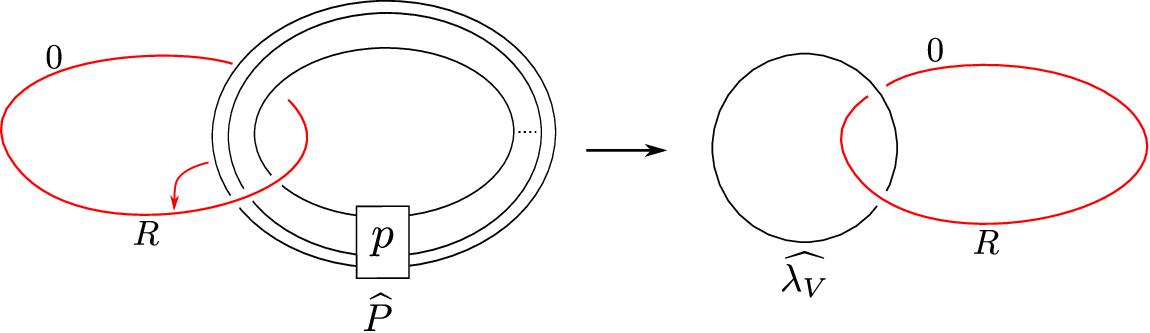}
\caption{(color online) A dualizable pattern $P$, and $\widehat{P}$. 
In this picture, $\mathbf{S}^{1}\times \mathbf{S}^{2}$ is represented by a knot $R$ with $0$-framing. 
Since $P$ is dualizable, $\widehat{P}$ is isotopic to $\widehat{\lambda_{V}}$ in $\mathbf{S}^{1}\times \mathbf{S}^{2}$. 
Such an isotopy is realized by isotopy on $\mathbf{S}^{3}$ and slidings over $R$. 
So, the algebraic winding number of $P$ is equal to the linking number of $R$ and $\widehat{\lambda_{V}}$. }
\label{figure:awinding}
\end{figure}
The following theorem is the most basic property of dualizable patterns. 
\begin{thm}[{e.g. \cite[Lemma~2.2]{Gompf-Miyazaki} and \cite[Theorem~2.8]{Miller-Piccirillo}}]\label{thm:dualizable-3dim}
If a pattern $P$ is dualizable, then 
there is an orientation-preserving homeomorphism $\phi\colon M_{P(U)}(0)\rightarrow  M_{P^{\ast}(U)}(0)$. 
\end{thm}
%
Miller and Piccirillo \cite{Miller-Piccirillo} extended Theorem~\ref{thm:dualizable-3dim} as follows. 
\begin{thm}[e.g.{\cite[Theorem~3.1]{Miller-Piccirillo}}]\label{thm:dualizable-4dim}
Let $P$ be a dualizable pattern. 
Then the homeomorphism $\phi\colon M_{P(U)}(0)\rightarrow  M_{P^{\ast}(U)}(0)$ in Theorem~\ref{thm:dualizable-3dim} extends to an orientation-preserving diffeomorphism $\Phi\colon X_{P(U)}(0)\rightarrow X_{P^{\ast}(U)}(0)$. 
\end{thm}
%

\subsection{From special annulus presentations to dualizable patterns}\label{sec:annulus-dualizable}
Miller and Piccirillo \cite[Section~5]{Miller-Piccirillo} constructed a dualizable pattern from a special annulus presentation. 
Here, we introduce their construction. 
\par
Let $K\subset \mathbf{S}^3$ be a knot with a special annulus presentation $(A, b)$ with $\partial A=c_1\cup c_2$. 
Recall that $A$ is a Hopf band since $(A,b)$ is special. 
Now, we consider the knot $A^{\pm 1}(K)$. 
In Figures~\ref{figure:annulus-dualizable1-2} and \ref{figure:annulus-dualizable3-4}, the left blue knots represent $K$, and each right black knot represents $A^{\pm1}(K)$ for the corresponding left $K$. 
More precisely, 
\begin{itemize}
\item in Figure~\ref{figure:annulus-dualizable1-2} $(1)$, $A$ is $-1$ twisted and the black knot is $A(K)$, 
\item in Figure~\ref{figure:annulus-dualizable1-2} $(2)$, $A$ is $-1$ twisted and the black knot is $A^{-1}(K)$, 
\item in Figure~\ref{figure:annulus-dualizable3-4} $(3)$, $A$ is $+1$ twisted and the black knot is $A^{-1}(K)$, 
\item in Figure~\ref{figure:annulus-dualizable3-4} $(4)$, $A$ is $+1$ twisted and the black knot is $A(K)$, 
\end{itemize} 
for each left $K$. 
Then, for each case, take a red curve $\beta \subset \mathbf{S}^{3}\setminus \nu(K)$ as in Figures~\ref{figure:annulus-dualizable1-2} and \ref{figure:annulus-dualizable3-4}. 
\begin{figure}[h]
\centering
\includegraphics[scale=0.73]{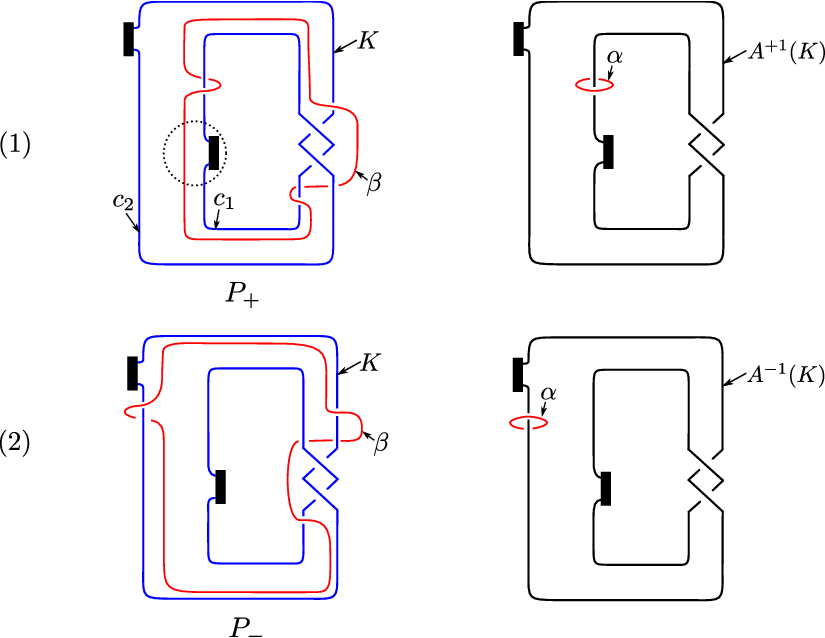}
\caption{(color online) From a special annulus presentation $(A,b)$ of a knot $K$ (the left blue knots) to dualizable patterns $P_{+}$ (top) and $P_{-}$ (bottom) given by $K\subset \mathbf{S}^{3}\setminus \nu(\beta)$. 
Here, the Hopf band $A$ is $-1$ twisted. 
The red curves $\beta$ run near $c_1$ for $(1)$ and near $c_2$ for $(2)$. 
Precisely, for $(1)$, there is a sufficiently thin annulus bounded by $\beta$ and $c_1$ such that the thin annulus does not intersect the band $b$ except the attaching regions and the thin annulus is on $A$ near the attaching region on $c_1$ (circled by the dotted circle). 
The right black knots represent $A^{+1}(K)$ for $(1)$ and $A^{-1}(K)$ for $(2)$. 
}\label{figure:annulus-dualizable1-2}
\end{figure}
\begin{figure}[h]
\centering
\includegraphics[scale=0.73]{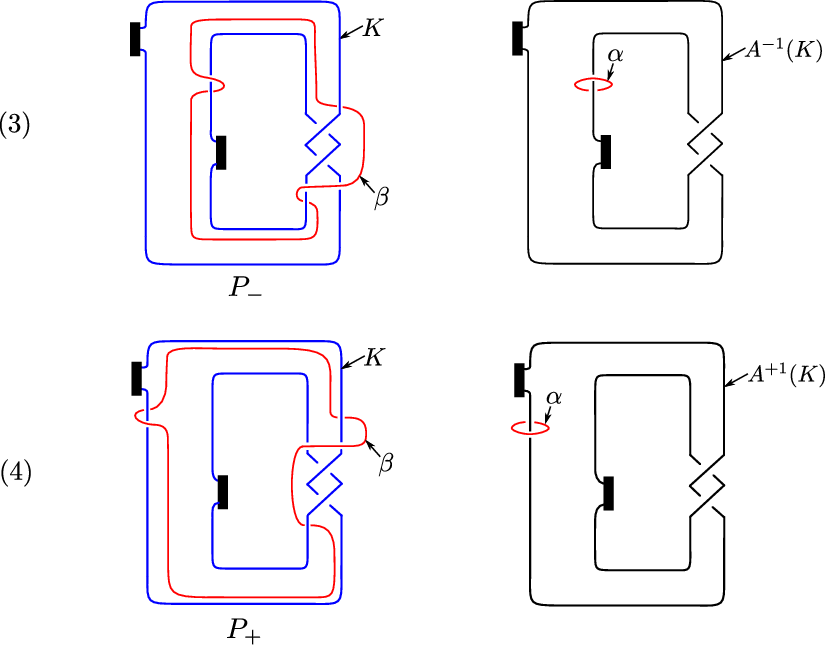}
\caption{(continuation of Figure~\ref{figure:annulus-dualizable1-2}) 
Here, the Hopf band $A$ is $+1$ twisted. 
The red curves $\beta$ run near $c_1$ for $(3)$ and near $c_2$ for $(4)$. 
The right black knots represent $A^{-1}(K)$ for $(3)$ and $A^{+1}(K)$ for $(4)$. 
}\label{figure:annulus-dualizable3-4}
\end{figure}
Note that 
the homeomorphism given in Figure~\ref{figure:Osoinach-homeo} induces a homeomorphism $\phi_{\pm1}\colon (M_{K}(0), \beta)\rightarrow (M_{A^{\pm1}(K)}(0), \alpha)$, where $\alpha \subset \mathbf{S}^{3}\setminus \nu(A^{\pm1}(K))$ is a meridian of $A^{\pm1}(K)$ and we regard $\beta$ and $\alpha$ as curves in $M_{K}(0)$ and $M_{A^{\pm1}(K)}(0)$, respectively. 
%
%
Note also that $V=\mathbf{S}^3\setminus \nu(\beta)$ is homeomorphic to a solid torus. 
Let $P_{\pm}$ be the pattern given by $K\subset V$ 
as the left pictures of $A^{\pm1}(K)$ in Figures~\ref{figure:annulus-dualizable1-2} and \ref{figure:annulus-dualizable3-4}, where we give a parameter of $V$ by regarding $V$ as a solid torus in a standard way so that $P_{\pm}(U)=K$. 
We give an orientation of $P_{\pm}$ arbitrarily. 
Then, we can prove that the patterns $P_{\pm}\subset V$ are dualizable as follows. 
Here, we only consider the case of Figure~\ref{figure:annulus-dualizable1-2} $(1)$. 
Similarly, for other three cases, we can prove that the patterns are dualizable. 
Firstly, draw $\widehat{P_{+}}$ in $M_{\beta}(0)=\mathbf{S}^1\times \mathbf{S}^2$ as in the first picture of Figure~\ref{figure:dualizable-isotopy}. 
Secondly, slide $\widehat{P_{+}}$ over the $0$-framed knot $\beta$ along the black arrow in the first picture of Figure~\ref{figure:dualizable-isotopy}. 
Then, we see that $\widehat{P_{+}}$ is isotopic to $\widehat{\lambda_{V}}$ in $M_{\beta}(0)$. 
Hence, by Proposition~\ref{prop:MP}, $P_{+}$ is dualizable and its dual $P_{+}^{\ast}$ is presented by the green curve $\widehat{\lambda_{V}}$ in $V^{\ast}=M_{\beta}(0)\setminus \nu(\widehat{P}_{+})$ in Figure~\ref{figure:dualizable-isotopy}. 
\begin{figure}[h]
\centering
\includegraphics[scale=0.75]{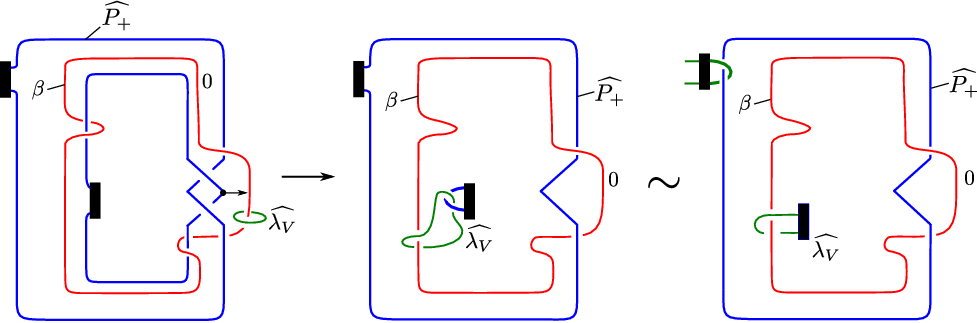}
\caption{(color online) By sliding $\widehat{P_{+}}$ over the $0$-framed knot $\beta$ (red curve) along the black arrow (in the left picture), we see that $\widehat{P_{+}}$ is isotopic to $\widehat{\lambda_{V}}$ in $\mathbf{S}^1\times \mathbf{S}^3$ (right picture). 
}\label{figure:dualizable-isotopy}
\end{figure}
\par
Obviously, $P_{\pm}(U)=K$ (as unoriented knots)
\footnote{
Miller and Piccirillo \cite[Section~5]{Miller-Piccirillo} took $\alpha$ as a meridian of $K$. So their dualizable patterns satisfy $P^{\ast}(U)=K$. 
On the other hand, our dualizable patterns satisfy $P_{\pm}(U)=K$. }. 
Miller and Piccirillo's proof \cite[Section~5]{Miller-Piccirillo} implies that $P_{+}^{\ast}(U)=A^{+1}(K)$ and $P_{-}^{\ast}(U)=A^{-1}(K)$ (as unoriented knots). 
For the sake of completeness, we introduce the proof. 
%
%
\par
As mentioned above, there is an orientation-preserving homeomorphism 
\[
\phi_{\pm1}\colon (M_{K}(0), \beta)\rightarrow (M_{A^{\pm1}(K)}(0), \alpha).
\]
Let $L \subset M_{A^{\pm1}(K)}(0)$ be the surgery dual to $A^{\pm1}(K)$. 
Since $\alpha$ is isotopic to $L$ in $M_{A^{\pm1}(K)}(0)$, we have 
\begin{align}
\mathbf{S}^{3}\setminus \nu(A^{\pm1}(K)) \label{eq1}
&\cong M_{A^{\pm1}(K)}(0)\setminus \nu(L)\\ \nonumber
&\cong M_{A^{\pm1}(K)}(0)\setminus \nu(\alpha)\\ \nonumber
&\cong \phi_{\pm1}^{-1}(M_{A^{\pm1}(K)}(0))\setminus \phi_{\pm1}^{-1}(\nu(\alpha))\\\nonumber
&= M_{K}(0)\setminus \nu(\beta)\\\nonumber
&\cong ((\mathbf{S}^3\setminus \nu(\beta))\setminus \nu(P_{\pm}))\cup_{\partial} (\mathbf{S}^{1}\times D^{2}), 
\end{align}
where the last union is given by identifying the $\lambda_{P_{\pm}}$ with a meridian of $\mathbf{S}^{1}\times D^{2}$. 
Recall that $V=\mathbf{S}^3\setminus \nu(\beta)$. 
Moreover, since $P_{\pm}$ is dualizable there is an orientation-preserving homeomorphism $f\colon V\setminus \nu(P_{\pm})\rightarrow V^{\ast}\setminus \nu(P_{\pm}^{\ast})$ such that $f(\lambda_{P_{\pm}})=\lambda_{V^{\ast}}$. 
Hence we obtain
\begin{align}
\mathbf{S}^{3}\setminus \nu(A^{\pm1}(K)) \label{eq2}
&\cong ((\mathbf{S}^3\setminus \nu(\beta))\setminus \nu(P_{\pm}))\cup_{\partial} (\mathbf{S}^{1}\times D^{2})\\ \nonumber
&=(V\setminus \nu(P_{\pm}))\cup_{\partial} (\mathbf{S}^{1}\times D^{2})\\ \nonumber
&\cong (V^{\ast}\setminus \nu(P_{\pm}^{\ast}))\cup_{\partial} (\mathbf{S}^{1}\times D^{2})\\ \nonumber
&\cong \mathbf{S}^{3}\setminus \nu(P_{\pm}^{\ast}(U)), 
\end{align}
where the last union is given by identifying $\lambda_{V^{\ast}}$ with a meridian of $\mathbf{S}^{1}\times D^{2}$. 
Finally, by the Knot Complement Theorem \cite{Gordon-Luecke}, we see that $A^{\pm1}(K)=P_{\pm}^{\ast}(U)$. 
As a consequence, we obtain Proposition~\ref{prop:Miller-Piccirillo}. 
%
\begin{prop}[{e.g. \cite[Proposition~5.3]{Miller-Piccirillo}}]\label{prop:Miller-Piccirillo}
Let $K$ be a knot with a special annulus presentation $(A,b)$. 
Then, there are dualizable patterns $P_{+}$ and $P_{-}$ such that $P_{\pm}(U)=K$ and $P_{\pm}^{\ast}(U)=A^{\pm 1}(K)$. 
In particular, such $P_{\pm}$ are drawn as in Figures~\ref{figure:annulus-dualizable1-2} and \ref{figure:annulus-dualizable3-4}. 
\end{prop}
In \cite{tagami7}, the author gave the explicit formula of $P^{\ast}_{\pm}$ for the dualizable pattern $P_{\pm}$ in Proposition~\ref{prop:Miller-Piccirillo}. 
\par
By Proposition~\ref{prop:Miller-Piccirillo}, we can regard Theorem~\ref{thm:dualizable-4dim} as an extension of Theorem~\ref{thm:AJOT}. 

%
\subsection{Oriented case}\label{sec:annulus-dualizable-oriented}
Let $K$ be an oriented knot with a special annulus presentation $(A,b)$. 
The orientation of $K$ induces orientations of $A^{\pm1}(K)$ naturally. 
We can also give the orientations induced by $K$ to 
the dualizable patterns $P_{\pm}$ constructed from $K$ as in Section~\ref{sec:annulus-dualizable} (see also Figures~\ref{figure:annulus-dualizable1-2} and \ref{figure:annulus-dualizable3-4}). 
Moreover, their duals $P^{\ast}_{\pm}$ are also oriented by the orientations of $P_{\pm}$.  
Then, we can check that $P^{\ast}_{\pm}(U)=-A^{\pm1}(K)$ as oriented knots as follows. 
Let $l$ be a longitude of $A^{\pm1}(K)$ with $\operatorname{lk}(l, A^{\pm1}(K))=0$. 
For convenience, orient $\beta$ depicted in Figures~\ref{figure:annulus-dualizable1-2} and \ref{figure:annulus-dualizable3-4} so that $\operatorname{lk}(\beta, K)=1$. 
Let $\mu_{\beta}$ be a meridian of $\beta$ and orient $\mu_{\beta}$ so that $\operatorname{lk}(\beta, \mu_{\beta})=1$. 
Then, we see that the composition of homeomorphisms $\mathbf{S}^{3}\setminus \nu(A^{\pm1}(K))\cong M_{A^{\pm1}(K)}(0)\setminus \nu(\alpha) \cong M_{K}(0)\setminus \nu(\beta)$ given in Equation (\ref{eq1}) sends $l$ to $\lambda_{P_{\pm}}-\mu_{\beta}$ (see Figure~\ref{figure:orientation}). 
Here, $\lambda_{P_{\pm}}$ bounds a disk in $M_{K}(0)\setminus \nu(\beta)$, and $\mu_{\beta}$ is $\lambda_{V}$ under the identification $V=\mathbf{S}^3\setminus \nu(\beta)$. 
Hence, $l$ is sent to $-\lambda_{V} $. 
Moreover, the homeomorphism $(V\setminus \nu(P_{\pm}))\cup_{\partial} (\mathbf{S}^{1}\times D^{2})\cong (V^{\ast}\setminus \nu(P_{\pm}^{\ast}))\cup_{\partial} (\mathbf{S}^{1}\times D^{2})$ in Equation~(\ref{eq2}) sends $\lambda_{V}$ to $\lambda_{P_{\pm}^{\ast}}$. 
As a consequence, $l$ is sent to a longitude of $P_{\pm}^{\ast}(U)$ with reversed orientation under the homeomorphism $\mathbf{S}^{3}\setminus \nu(A^{\pm1}(K))\cong \mathbf{S}^{3}\setminus \nu(P_{\pm}^{\ast}(U))$ in Equation~(\ref{eq2}), and we obtain the following. 
\begin{prop}[{the oriented version of Proposition~\ref{prop:Miller-Piccirillo}}]\label{prop:Miller-Piccirillo-oriented}
Let $K$ be an oriented knot with a special annulus presentation $(A,b)$. 
Give $A^{\pm1}(K)$ the orientation induced by $K$. 
Then, there are dualizable patterns $P_{+}$ and $P_{-}$ such that $P_{\pm}(U)=K$ and $P_{\pm}^{\ast}(U)=-A^{\pm 1}(K)$ as oriented knots. 
In particular, such $P_{\pm}$ are drawn as in Figures~\ref{figure:annulus-dualizable1-2} and \ref{figure:annulus-dualizable3-4}, where the orientations of $P_{\pm}$ are induced by $K$ naturally. 
\end{prop}
%
%
%
\begin{figure}[h]
\centering
\includegraphics[scale=0.9]{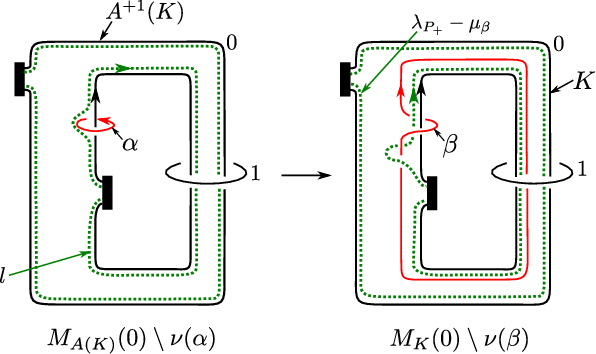}
\caption{(color online) A longitude $l$ of $A(K)$ in $\mathbf{S}^{3}\setminus\nu(A(K))\cong M_{A(K)}(0)\setminus \nu(\alpha)$ (the dotted green loop in the left). 
Here, we consider the case $A$ is $-1$ twisted.  
The longitude $l$ is sent to $\lambda_{P_{+}}-\mu_{\beta}$ under the homeomorphism $\mathbf{S}^{3}\setminus \nu(A^{\pm1}(K)) \cong M_{K}(0)\setminus \nu(\beta)$. 
}\label{figure:orientation}
\end{figure}
%
%
%
\section{Piccirillo's RGB-diagram}\label{sec:RGB}
In this section, we recall Piccirillo's work \cite{Piccirillo} and describe a relation between the work and a special annulus presentation. 
%
\subsection{RGB-diagram}\label{sec:RGB-diagram}
Let $L$ be a Kirby diagram which consists of one $1$-handle $R$ represented by a dotted unknotted circle and two $2$-handles $G$ and $B$ represented by $0$-framed knots. 
Suppose that $R$, $G$ and $B$ satisfying the following: 
\begin{itemize}
\item $G\cup R$ is isotopic to $G\cup \mu_{G}$ in $\mathbf{S}^3$, where $\mu_{G}$ is a meridian of $G$, 
\item $B\cup R$ is isotopic to $B\cup \mu_{B}$ in $\mathbf{S}^3$, where $\mu_{B}$ is a meridian of $B$, and 
\item the linking number $\operatorname{lk}(G,B)$ of $G$ and $B$ is zero. 
\end{itemize}
Then, we call $L$ an {\it RGB-diagram}. 
For example, see the bottom right picture in Figure~\ref{figure:ex} and 
\cite{Piccirillo} where $R$, $G$ and $B$ are drawn as red, green and blue curves, respectively. 
For an RGB-diagram $L=R\cup G\cup B$, we can construct two knots $K_{G}$ and $K_{B}$ as follows. 
Let $D_{R}$ be a disk bounded by $R$. 
Since $G\cup R$ is isotopic to $G\cup \mu_{G}$ in $\mathbf{S}^3$, we can isotope $L$ so that $G\cap D_{R}$ is a single point. 
Then, slide $B$ over $G$ as needed to remove the all points in $B\cap D_{R}$, and denote the resulting knot (obtained from $B$) by $K_{B}$. 
Note that the framing of $K_{B}$ is $0$ because the framings of $B$ and $G$ are zero and $\operatorname{lk}(G,B)$ is zero. 
Moreover, after the slide, we can cancel the $1$-handle and the $2$-handle represented by $R$ and $G$, respectively. 
Hence, we have $X_{L}\cong X_{K_{B}}(0)$. 
Reversing the roles of $G$ and $B$, we obtain $K_{G}$ and $X_{L}\cong X_{K_{G}}(0)$. 
As a consequence, we obtain the following. 
\begin{thm}[{\cite[Theorem~2.1]{Piccirillo}}]\label{thm:Piccirillo}
Let $L$ be an RGB-diagram and $K_{G}$ and $K_{B}$ be as above. 
Then we have $X_{L}\cong X_{K_{G}}(0)\cong X_{K_{B}}(0)$. 
\end{thm}
\begin{rem}
Piccirillo \cite{Piccirillo} denotes $K_{B}$ by $K$ and $K_{G}$ by $K'$. 
\end{rem}
Piccirillo \cite[Proposition~4.2]{Piccirillo} explained a relation between dualizable patterns and RGB-diagrams as follows. 
\begin{prop}[{\cite[Proposition~4.2]{Piccirillo}}]\label{prop:Piccirillo}
For any RGB-diagram $L=R\cup G \cup B$, there exists a dualizable pattern $P$ such that $P(U)=K_{B}$ and $P^{\ast}(U)=K_{G}$. 
Conversely, for any dualizable pattern $P$, there exists an RGB-diagram $L=R\cup G \cup B$ such that $K_{B}=P(U)$ and $K_{G}=P^{\ast}(U)$. 
\end{prop}
One of the purposes of this manuscript is drawing an RGB-diagram from a given special annulus presentation through the dualizable pattern given in Section~\ref{sec:annulus-dualizable}. 
So, we recall the proof of the second claim of Proposition~\ref{prop:Piccirillo}. 
\begin{lem}\label{lem:pattern-to-RGB}
For any dualizable pattern $P$, there exists an RGB-diagram $L=R\cup G \cup B$ such that $K_{B}=P(U)$ and $K_{G}=P^{\ast}(U)$. 
\end{lem}
\begin{proof}
This proof is due to Piccirillo \cite[Proposition~4.2]{Piccirillo}. 
Let $P$ be a dualizable pattern. Draw $P$ as Figure~\ref{figure:pattern}. 
Let $J=R\cup G\cup B'$ be the right Kirby diagram depicted in Figure~\ref{figure:diagramJ}. 
Then, we see that $X_{P(U)}(0)\cong X_{J}$. 
Note that $J$ is not an RGB-diagram generally because $R$ may not be a meridian of $B'$. 
However, because $P$ is dualizable, the framed knot $B'$ can be made a meridian of $R$ by sliding $B'$ over $R$ finitely many times (see Proposition~\ref{prop:MP} and the caption of Figure~\ref{figure:awinding}). 
Denote the resulting framed knot by $B''$. 
In general, the linking number of $G$ and $B''$ is not zero. 
So, we deform $B''$ by sliding $B''$ over $R$ as in Figure~\ref{figure:slide1} in order to vanish the linking number. 
Since the linking number of $R$ and $G$ is $\pm 1$, we can finish such deformations finitely many times. 
We denote the resulting framed knot by $B$ and its framing by $a\in \mathbf{Z}$. 
Put $L=R\cup G\cup B$. 
Then, we have $a=0$ as follows. 
Let $\widetilde{K_{B}}$ be the knot obtained from $B$ of $L$ in the same way as the construction of $K_B$ from an RGB-diagram. 
Then, we have $X_{P(U)}(0)\cong X_{J}\cong X_{L}\cong X_{\widetilde{K_{B}}}(a)$. 
Comparing the signatures of $X_{P(U)}(0)$ and $X_{\widetilde{K_{B}}}(a)$, we obtain $a=0$. 
As a result, we see that $L=R\cup G\cup B$ is an RGB-diagram and $X_{P(U)}(0)\cong X_{K_{B}}(0)$. 
\par
We will prove that $L$ is the desired RGB-diagram. 
Let $g\colon X_{P(U)}(0) \rightarrow X_{K_{B}}(0)$ be the diffeomorphism given as above. 
By restricting $g$ to the boundary, we have $g|_{\partial}\colon M_{P(U)}(0)\rightarrow M_{K_{B}}(0)$ which sends the surgery dual to $P(U)$ to the surgery dual to $K_{B}$ 
(in fact, these surgery duals are sent to $B'$ and $B''$ by the differomorphisms $X_{P(U)}(0)\cong X_{J}\cong X_{L}\cong X_{K_{B}}(0)$). 
Hence, by considering the inverses of the $0$-framed surgeries, we see that $g|_{\partial}$ induces an orientation-preserving homeomorphism $f\colon \mathbf{S}^{3}\rightarrow \mathbf{S}^{3}$ which sends $P(U)$ to $K_{B}$, that is, $P(U)=K_{B}$. 
To prove $P^{\ast}(U)=K_{G}$, define a Kirby diagram $J^{\ast}=R\cup G'^{\ast}\cup B^{\ast}$ as in the right picture of Figure~\ref{figure:diagramJdual}. 
We see that $X_{P^{\ast}(U)}(0)\cong X_{J^{\ast}}$. 
Since $P$ is dualizable, by Proposition~\ref{prop:MP} (and Figure~\ref{figure:dualizable}), the components $G\cup B'$ of $J$ is isotopic to the components $G'^{\ast}\cup B^{\ast}$ of $J^{\ast}$ as framed links in the boundary of the $1$-handle represented by $R$, where the isotopy sends $G$ to $G'^{\ast}$ and $B'$ to $B^{\ast}$ (see Figure~\ref{figure:diagramJandJdual}). 
Hence, we obtain $X_{P^{\ast}(U)}(0)\cong X_{J^{\ast}}\cong X_{J}\cong X_{L}\cong X_{K_{G}}(0)$. 
Note that the last diffeomorphism is given by Theorem~\ref{thm:Piccirillo}. 
Hence, by the same discussion as above, there is a differomorphism $g^{\ast} \colon X_{P^{\ast}(U)}(0) \rightarrow X_{K_{G}}(0)$ and its restriction to the boundaries 
$g^{\ast}|_{\partial}\colon M_{P^{\ast}(U)}(0)\rightarrow M_{K_{G}}(0)$ induces an orientation-preserving homeomorphism $f^{\ast}\colon \mathbf{S}^{3}\rightarrow \mathbf{S}^{3}$ which sends $P^{\ast}(U)$ to $K_{G}$, that is, $P^{\ast}(U)=K_{G}$. 
\end{proof}
\begin{figure}[h]
\centering
\includegraphics[scale=0.45]{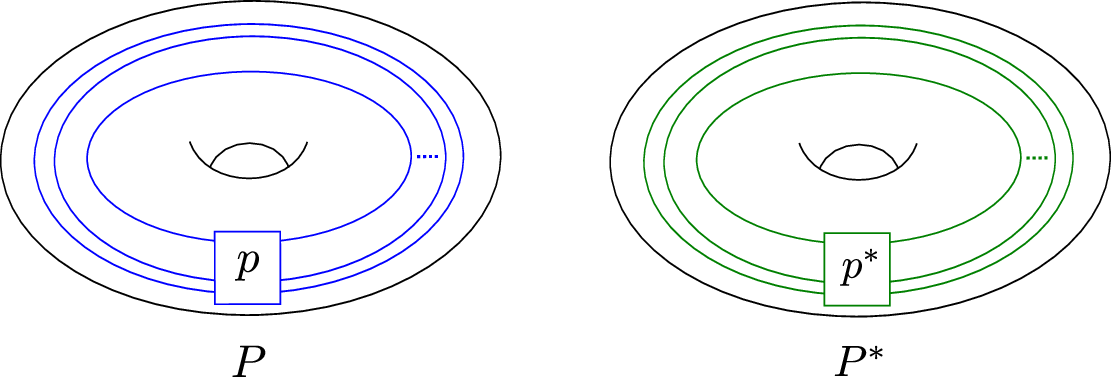}
\caption{(color online) Patterns $P$ and $P^{\ast}$. 
}\label{figure:pattern}
\end{figure}
\begin{figure}[h]
\centering
\includegraphics[scale=0.58]{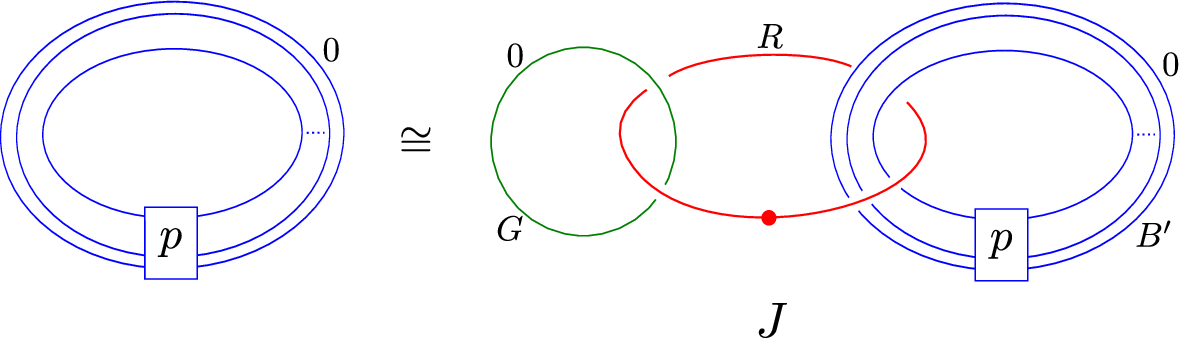}
\caption{(color online) $X_{P(U)}(0)\cong X_{J}$
}\label{figure:diagramJ}
\end{figure}
\begin{figure}[h]
\centering
\includegraphics[scale=0.6]{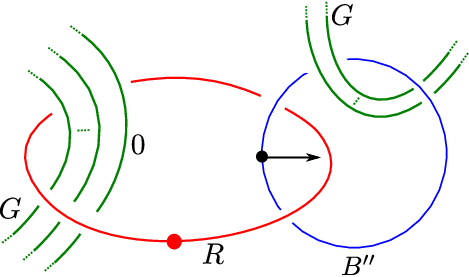}
\caption{(color online) Slide to vanish the linking number of $G$ and $B''$
}\label{figure:slide1}
\end{figure}
\begin{figure}[h]
\centering
\includegraphics[scale=0.58]{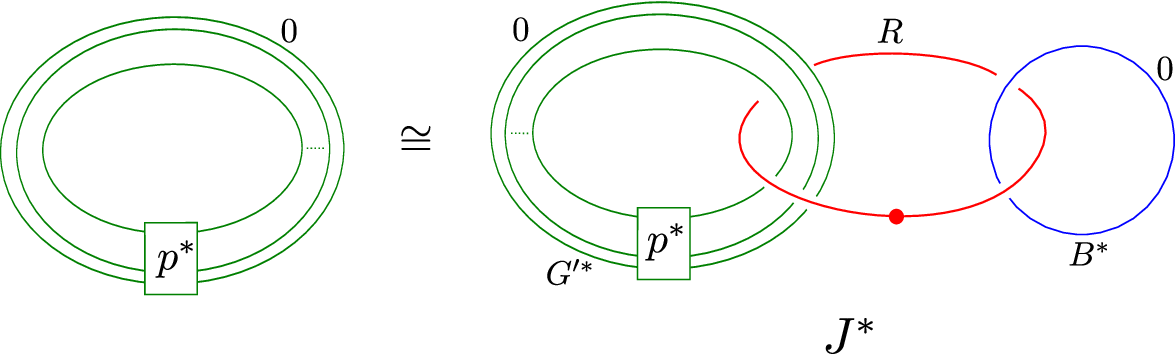}
\caption{(color online) $X_{P^{\ast}(U)}(0)\cong X_{J^{\ast}}$
}\label{figure:diagramJdual}
\end{figure}
\begin{figure}[h]
\centering
\includegraphics[scale=0.37]{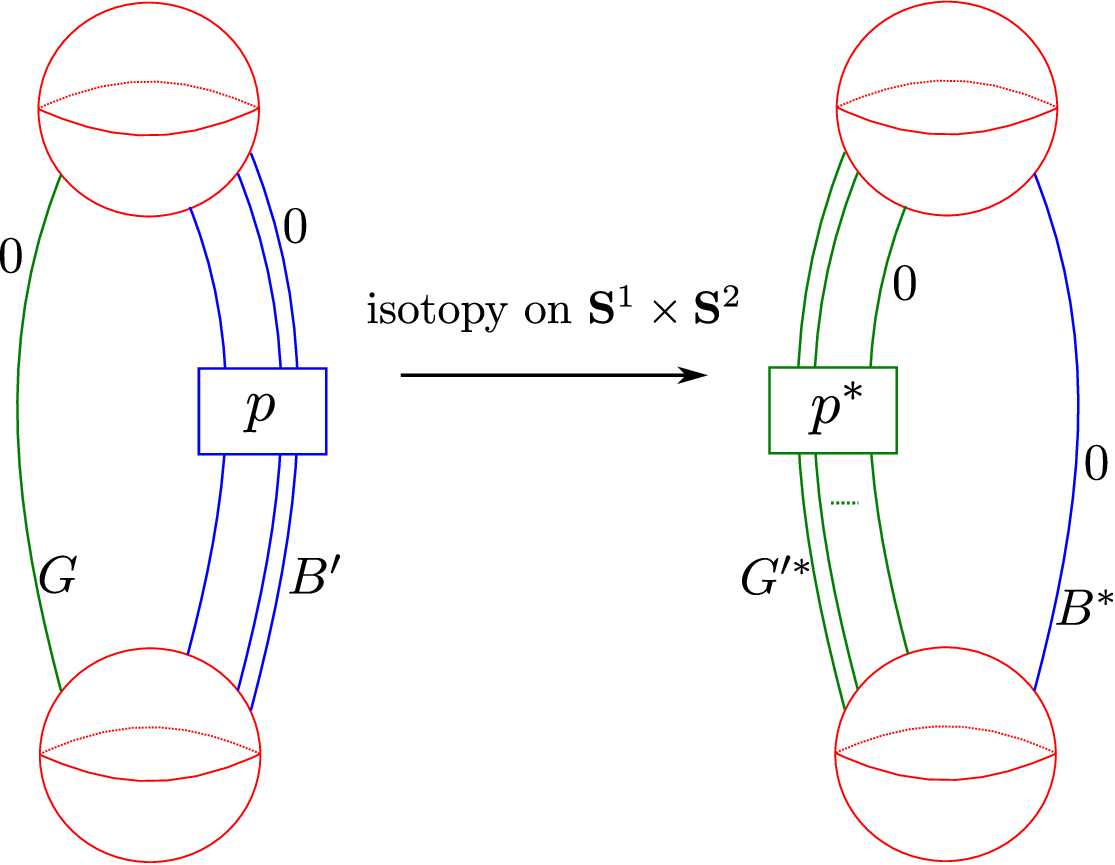}
\caption{(color online) $X_{J}\cong X_{J^{\ast}}$
}\label{figure:diagramJandJdual}
\end{figure}
\begin{rem}
We see that Figure~\ref{figure:diagramJandJdual} gives an alternative proof of Theorem~\ref{thm:dualizable-4dim}. 
\end{rem}
\begin{rem}
For a pattern $P\colon \mathbf{S}^1\rightarrow V$, denote the pattern obtained by twisting $n$ times along a meridian of $V$ by $\tau_{n}(P)$. 
By utilizing the technique in Figure~\ref{figure:diagramJandJdual}, we can prove that $X_{P(U)}(n)\cong X_{\tau_{n}(P^{\ast})(U)}(n)$ for any $n\in \mathbf{Z}$ (see Figure~\ref{figure:diagramJandJdual2} and see also \cite[Theorem~3.6]{Miller-Piccirillo}). 
Moreover, we can prove that $X_{\tau_{m}(P)(U)}(m+n)\cong X_{\tau_{n}(P^{\ast})(U)}(m+n)$ (see Figure~\ref{figure:diagramJandJdual3}). 
\par
Baker and Motegi \cite[Theorem~2.5]{Baker-Motegi} constructed pairs of knots $K_{m}$ and $k_{n}$ which satisfy $M_{K_{m}}(m+n) \cong M_{k_{n}}(m+n)$ for any $m,n\in\mathbf{Z}$. 
We remark that the pair $\{ \tau_{m}(P)(U), \tau_{n}(P^{\ast})(U)\}$ is essentially equal to the pair $\{K_{m}, k_{n}\}$. 
\end{rem}
\begin{figure}[ht]
\centering
\includegraphics[scale=0.37]{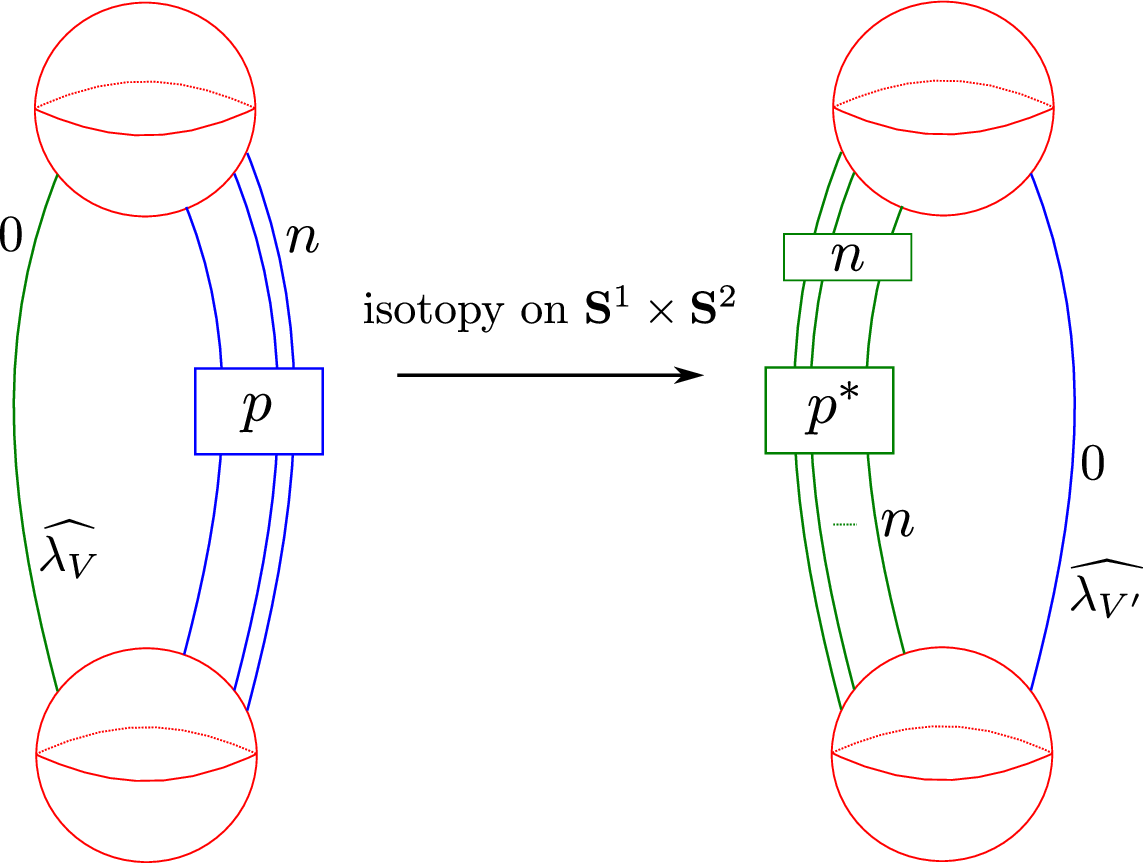}
\caption{(color online) The left blue curve with framing $n$ is $P\subset V$ and the right green curve with framing $n$ is $\tau_{n}(P^{\ast})\subset V'$. 
We see that $P(U)$ and $\tau_{n}(P^{\ast})(U)$ have the diffeomorphic $n$-trace. The box with the label $n$ means the $n$ full-twists. 
}\label{figure:diagramJandJdual2}
\end{figure}
\begin{figure}[h]
\centering
\includegraphics[scale=0.91]{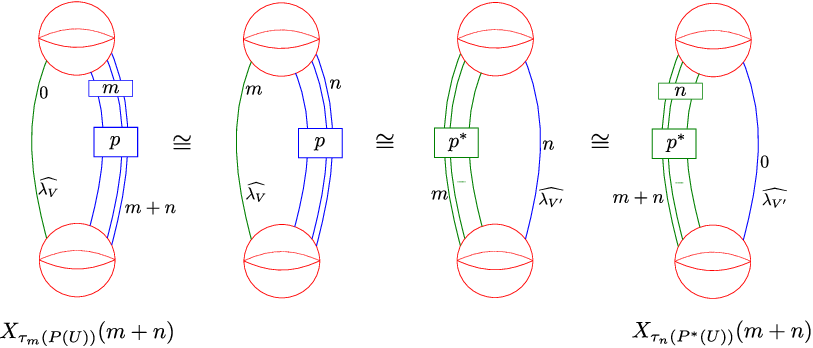}
\caption{(color online) $X_{\tau_{m}(P)(U)}(m+n)\cong X_{\tau_{n}(P^{\ast})(U)}(m+n)$
}\label{figure:diagramJandJdual3}
\end{figure}
%
\subsection{Oriented RGB-diagram}\label{sec:RGB-diagram-oriented}
Recall that dualizable patterns are oriented. 
We can also consider orientations of RGB-diagrams as follows. 
An RGB-diagram $L=R\cup G\cup B$ is {\it oriented} if $G$ and $B$ are oriented so that they are homologous in $\mathbf{S}^{3}\setminus R$. 
From an oriented RGB-diagram $L$, we can also construct two knots $K_{B}$ and $K_{G}$ as in Section~\ref{sec:RGB-diagram}. 
Then, we can give orientations of $K_{B}$ and $K_{G}$ by the orientations of $B$ and $G$, respectively. 
\par 
Let $P$ be a dualizable pattern. 
In order to construct an oriented RGB-diagram $L=R\cup G\cup B$ from $P$, recall the proof of Lemma~\ref{lem:pattern-to-RGB}. 
Firstly, we construct a Kirby diagram $J=R\cup G\cup B'$ from $P$ (see Figure~\ref{figure:diagramJ}). 
Then, we can orient $G$ and $B'$ by using the orientations of $\lambda_{V}$ and $P$, respectively. 
Note that $G$ and $B'$ are homologous in $\mathbf{S}^3\setminus R$ because of the definition of the orientation of $\lambda_{V}$. 
Secondly, to construct $L=R\cup G\cup B$ from $J$, we slide $B'$ over $R$ finitely many times. 
After the operation, the linking number of $B'$ and $R$ does not change. 
Hence, $G$ and $B$ are homologous in $\mathbf{S}^3\setminus R$, and $L$ is an oriented RGB-diagram. 
\par 
Recall that the orientations of $B$ and $G$ are given by the orientations of $P$ and $\lambda_{V}$, respectively. 
Moreover, the orientation of $P^{\ast}$ is given by the orientation of $\lambda_{V}$. 
Hence, by Lemma~\ref{lem:pattern-to-RGB}, we see that $K_{B}=P(U)$ and $K_{G}=P^{\ast}(U)$ as oriented knots. 
\begin{lem}[{the oriented version of Lemma~\ref{lem:pattern-to-RGB}}]\label{lem:pattern-to-RGB-oriented}
For any dualizable pattern $P$, there exists an oriented RGB-diagram $L=R\cup G \cup B$ such that $K_{B}=P(U)$ and $K_{G}=P^{\ast}(U)$ as oriented knots. 
\end{lem}
%
%
\subsection{From special annulus presentations to RGB-diagrams}\label{sec:annulus-to-RGB}
Let $K$ be a knot with a special annulus presentation $(A, b)$. 
In Section~\ref{sec:annulus-dualizable}, we construct a dualizable patterns $P_{\pm}$ such that $P_{\pm}(U)=K$ and $P_{\pm}^{\ast}(U)=A^{\pm1}(K)$. 
In this subsection, we draw RGB-diagrams $L_{\pm}=R\cup G_{\pm} \cup B_{\pm}$ such that $K_{B_{\pm}}=P_{\pm}(U)=K$ and $K_{G_{\pm}}=P_{\pm}^{\ast}(U)=A^{\pm 1}(K)$ by utilizing (the proof of)  Lemma~\ref{lem:pattern-to-RGB}. 
Here, we only consider the case of Figure~\ref{figure:annulus-dualizable1-2} $(1)$. 
For three other cases, the similar discussions work. 
\par 
Let $K$ be a knot with a special annulus presentation $(A, b)$, where $A$ is $-1$ twisted (see the left picture of Figure~\ref{figure:annulus-RGB}). 
The dualizable pattern $P=P_{+}\subset \mathbf{S}^{3}\setminus \nu(\beta)$, which satisfies $P(U)=K$ and $P^{\ast}(U)=A^{+1}(K)$, is given as the center picture of Figure~\ref{figure:annulus-RGB}. 
Then, the Kirby diagram $J=R\cup G\cup B'$ obtained from $P$ as in the proof of Lemma~\ref{lem:pattern-to-RGB} is given as the right picture of Figure~\ref{figure:annulus-RGB}. 
To obtain the RGB-diagram $L$ given in the proof of Lemma~\ref{lem:pattern-to-RGB}, firstly slide $B'$ over $R$ along the black arrow as in the left picture of Figure~\ref{figure:annulus-RGB2} and denote the resulting knot by $B''$. 
In the resulting diagram, the linking number of $G$ and $B''$ is not zero. 
So, by sliding $B''$ over $R$ along the dotted black arrow as in the center picture of Figure~\ref{figure:annulus-RGB2}, we delete the linking number. 
Then the resulting diagram is an RGB-diagram $L=R\cup G\cup B$ (see the right picture of Figure~\ref{figure:annulus-RGB2}). 
This is the desired RGB-diagram $L_{+}=R\cup G_{+}\cup B_{+}$. 
As a consequence, we obtain Theorem~\ref{thm:annulus-RGB} below. 
%
\begin{thm}\label{thm:annulus-RGB}
Let $K$ be a knot with a special annulus presentation $(A,b)$. 
Then, there are dualizable patterns $P_{\pm}$ and RGB-diagrams $L_{\pm}=R\cup G_{\pm}\cup B_{\pm}$ such that $K=P_{\pm}(U)=K_{B_{\pm}}$ and $A^{\pm}(K)=P_{\pm}(U)=K_{G_{\pm}}$. 
In particular, such $P_{\pm}$ and $L_{\pm}=R\cup G_{\pm}\cup B_{\pm}$ are given as in Figure~\ref{figure:annulus-dualizable-RGB}. 
\end{thm}
%
By the discussions in Sections~\ref{sec:annulus-dualizable-oriented} and \ref{sec:RGB-diagram-oriented}, we obtain the oriented version of Theorem~\ref{thm:annulus-RGB} as follows. 
\begin{thm}[the oriented version of Theorem~\ref{thm:annulus-RGB}]\label{thm:annulus-RGB-oriented}
Let $K$ be an oriented knot with a special annulus presentation $(A,b)$. 
Give $A^{\pm1}(K)$ the orientation induced by $K$. 
Then, there are dualizable patterns $P_{\pm}$ and oriented RGB-diagrams $L_{\pm}=R\cup G_{\pm}\cup B_{\pm}$ such that $K=P_{\pm}(U)=K_{B_{\pm}}$ and $A^{\pm}(K)=-P_{\pm}(U)=-K_{G_{\pm}}$. 
In particular, such $P_{\pm}$ and $L_{\pm}=R\cup G_{\pm}\cup B_{\pm}$ are given as in Figure~\ref{figure:annulus-dualizable-RGB}, where the orientations of $P_{\pm}$ and $B_{\pm}$ are induced by $K$ and the orientations of $P_{\pm}^{\ast}$ and $G_{\pm}$ are induced by a meridian of $\beta$ (or $R$) which is homologous to $K$ in $\mathbf{S}^{3}\setminus \beta$. 
\end{thm}
%
%
\begin{figure}[h]
\centering
\includegraphics[scale=0.7]{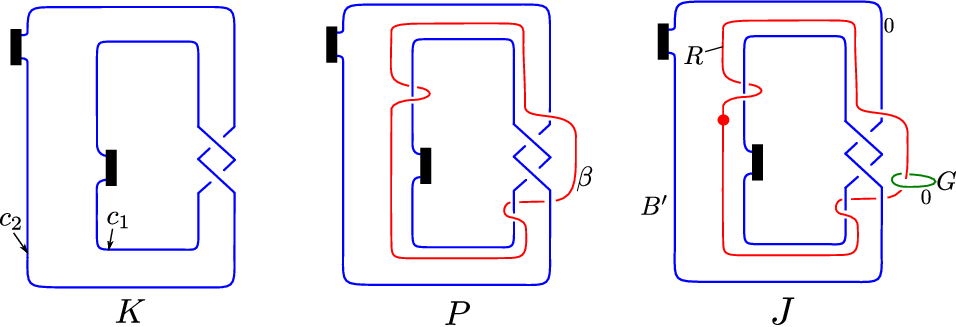}
\caption{(color online) A knot $K$ with a special annulus presentation (left), a dualizable pattern $P\subset \mathbf{S}^{3}\setminus \nu(\beta)$ satisfying $P(U)=K$ and $P^{\ast}(U)=A^{+1}(K)$ (center) and the Kirby diagram $J$ corresponding to $P$ (right). 
}\label{figure:annulus-RGB}
\end{figure}
\begin{figure}[h]
\centering
\includegraphics[scale=0.72]{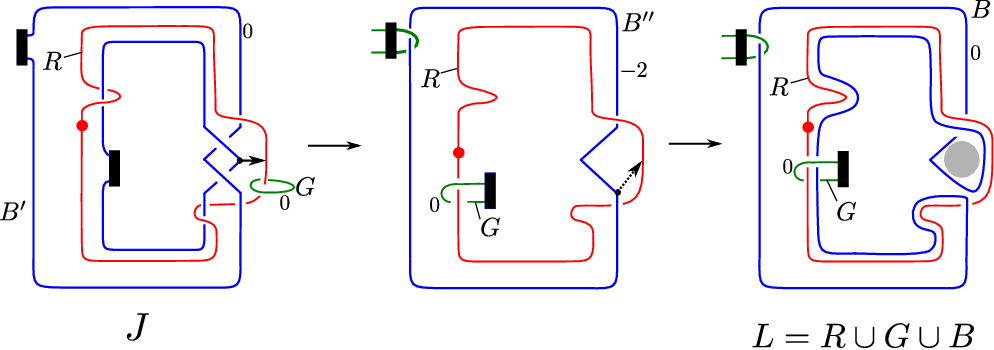}
\caption{(color online) From $J$ to the RGB-diagram $L$. 
In the gray area in the right picture, the band (the green component $G$) may appear. 
}\label{figure:annulus-RGB2}
\end{figure}
\begin{figure}[p]
\centering
\includegraphics[scale=0.7]{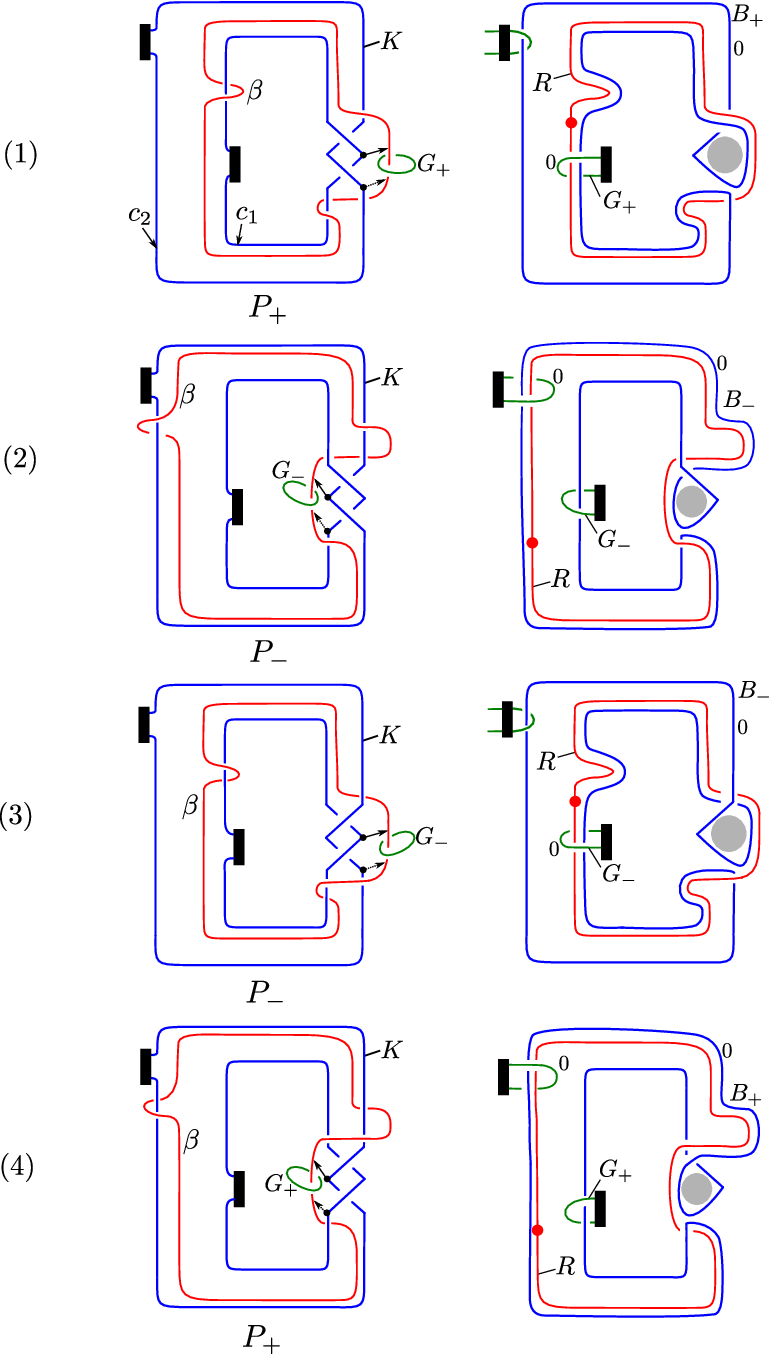}
\caption{(color online) 
The left blue curves represent knots $K$ which have special annulus presentations. 
The patterns $P_{\pm}$ are given by $K$ in $\mathbf{S}^{3}\setminus \nu(\beta)=V$, where the parameters of $V\cong \mathbf{S}^1\times D^{2}$ are given so that $P_{\pm}(U)=K$. 
The right pictures are $L_{\pm}=R\cup G_{\pm}\cup B_{\pm}$. 
The RGB-diagrams $L_{\pm}$ are obtained from the left pictures by (i) replacing $\beta$ with dotted circles $R$, 
(ii) regarding $K$ and the green curves $G_{\pm}$ as $0$-framed $2$-handles, (iii) sliding $K$ over $R$ along the black arrows and (iv) sliding $K$ over $R$ along the dotted black arrows. 
In the gray areas in the right pictures, the band (the green component $G_{\pm}$) may appear. 
}\label{figure:annulus-dualizable-RGB}
\end{figure}
%
%
%
\begin{rem}\label{rem:inverse}
We remark that there are a dualizable pattern and an RGB-diagram which do not arising from any special annulus presentation. 
In fact, the four-ball genera $g_4$ of knots with annulus presentations are less than $2$. 
On the other hand, Piccirillo \cite[Example~3.4]{Piccirillo} gave an RGB-diagram whose $K_{G}$ has $g_{4}(K_{G})= 2$. 
Moreover, by Proposition~\ref{prop:Piccirillo}, we can construct a dualzable pattern $P$ which satisfies $g_{4}(P^{\ast}(U))= 2$. 
\end{rem}
\begin{rem}\label{rem:unknotting-one}
Let $K$ be an unknotting number one knot. 
It is know that such a knot $K$ has a special annulus presentation (see \cite[Lemma~2.2]{AJOT}). 
Then, by Theorem~\ref{thm:annulus-RGB}, we obtain RGB-diagrams $L_{\pm}=R\cup G_{\pm}\cup B_{\pm}$ from the special annulus presentation. 
On the other hand, Piccirillo \cite{Piccirillo2} constructed an RGB-diagram from an unknotting number one knot. 
We see that the RGB-diagram is equal to $L_+$ or $L_{-}$.  
\end{rem}
\begin{ex}
Consider the special annulus presentation of $6_3$ given in Figure~\ref{figure:annulus-pre} (and see also Figure~\ref{figure:annulus-twist}). 
By applying Theorem~\ref{thm:annulus-RGB} and Figure~\ref{figure:annulus-dualizable-RGB} $(4)$, 
we obtain an RGB-diagram $L=R\cup G\cup B$ from the the special annulus presentation, which satisfies $K_{G}=A(6_3)$ and $K_{B}=6_3$ (see the bottom right picture in Figure~\ref{figure:ex}). 
\end{ex}
\begin{figure}[h]
\centering
\includegraphics[scale=0.74]{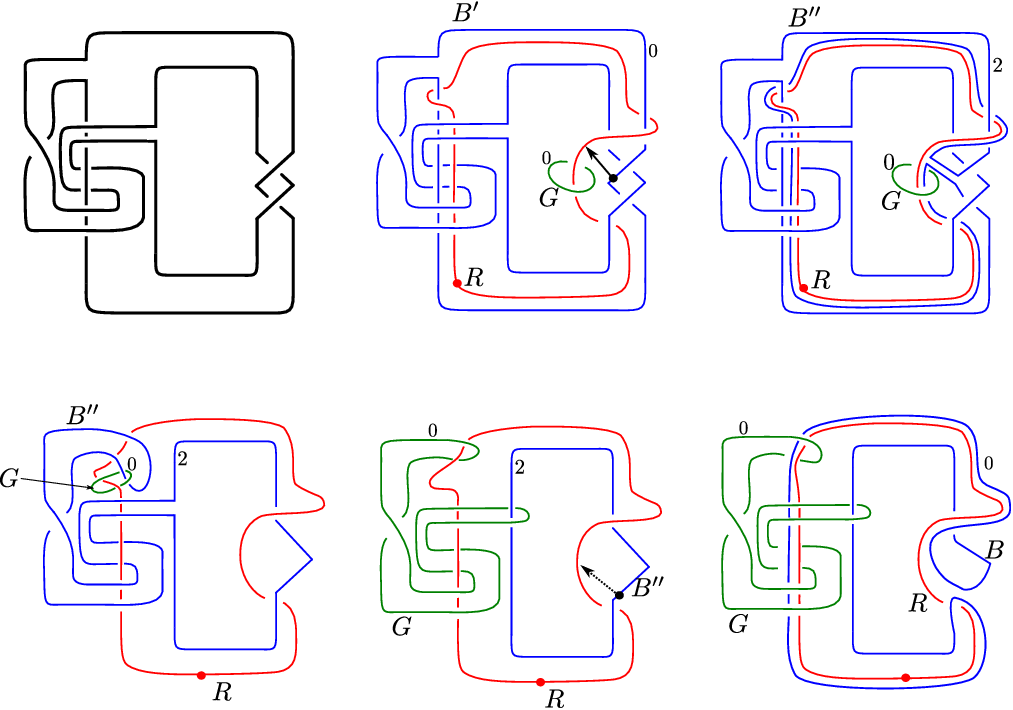}
\caption{
(color online)
A special annulus presentation (top left). 
A Kirby diagram $J$ in Section~\ref{sec:annulus-to-RGB} (top center). 
The bottom left and center pictures are isotopic to the top right picture. 
}\label{figure:ex}
\end{figure}
%
\section{Application}\label{sec:application}
In this section, as an application of the discussion in Section~\ref{sec:annulus-to-RGB}, we introduce a sufficient condition for a knot with a special annulus presentation $(A,b)$ to be either $K=A(K)$ or $K=A^{-1}(K)$. 
\begin{thm}\label{thm:trivial}
Let $K$ be an oriented knot with a special annulus presentation $(A, b)$ with $\partial A=c_1\cup c_2$. 
Give $A^{\pm1}(K)$ the orientation induced by $K$. 
Let $D_{i}$ be a disk bounded by $c_i$ for $i=1,2$. 
Then if $\operatorname{Int}(D_{i})\cap b=\emptyset$ for some $i$, we obtain either $A(K)=K$ or $A^{-1}(K)=K$ as oriented knots. 
\end{thm}
\begin{proof}
We only consider the case where $A$ is $-1$ twisted and $D_{1}\cap b=\emptyset$. 
Then, for the oriented RGB-diagram $L_{+}=R\cup G_{+}\cup B_{+}$ given in Theorem~\ref{thm:annulus-RGB-oriented}, there is a disk $D_{R}$ bounded by $R$ such that the intersections $D_{R}\cap G_{+}$ and $D_{R}\cap B_{+}$ consist of exactly one point (see the center of Figure~\ref{figure:trivial-annulustwist}, where the points in $D_{R}\cap G_{+}$ and $D_{R}\cap B_{+}$ are drawn by black dots). 
By the definition of $K_{G_{+}}$ (see Section~\ref{sec:RGB-diagram}), the knot $K_{G_{+}}$ is obtained from $G_{+}$ by sliding $G_{+}$ over $B_{+}$ along the black arrow in the right picture in Figure~\ref{figure:trivial-annulustwist}. 
Similarly, the knot $K_{B_{+}}$ is obtained from $B_{+}$ by sliding $B_{+}$ over $G_{+}$ along the dotted black arrow in the right picture in Figure~\ref{figure:trivial-annulustwist}. 
Hence, we obtain $K_{G_{+}}=K_{B_{+}}$ as unoriented knots. 
Moreover, since $B_{+}$ and $G_{+}$ are homologous in $\mathbf{S}^{3}\setminus R$, we have $-K_{G_{+}}=K_{B_{+}}$ as oriented knots. 
By Theorem~\ref{thm:annulus-RGB-oriented}, we see that $A^{+1}(K)=-K_{G_{+}}=K_{B_{+}}=K$. 
\end{proof}
\begin{figure}[h]
\centering
\includegraphics[scale=0.8]{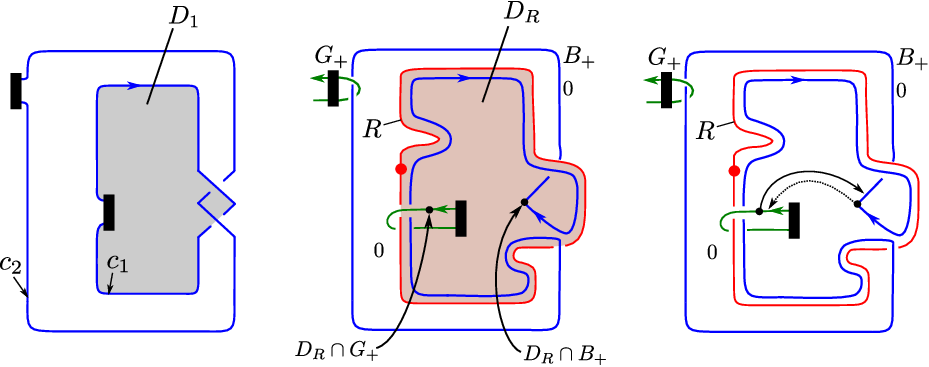}
\caption{(color online) In the gray area $D_{1}$, the band does not appear (left). 
The intersections $D_{R}\cap G_{+}$ and $D_{R}\cap B_{+}$ are only one point (center). 
}\label{figure:trivial-annulustwist}
\end{figure}
\begin{rem}
It is easy seen that a non-trivial knot satisfying the hypothesis of Theorem~\ref{thm:trivial} has unknotting number one. 
Conversely, a knot of unknotting number one has an annulus presentation which satisfies the hypothesis of Theorem~\ref{thm:trivial} 
by the construction given in \cite[Lemma~2.2]{AJOT}. 
\end{rem}
\begin{rem}
After submitting the first draft of this manuscript to arXiv, in \cite[Section~4.2]{Manolescu-Piccirillo}, Manolescu and Piccirillo introduced a result similar to Theorem~\ref{thm:trivial}.  
\end{rem}
%
\begin{ex}
Consider the special annulus presentation of $6_3$ given in Figure~\ref{figure:annulus-pre}. 
We see that this annulus presentation satisfies the condition of Theorem~\ref{thm:trivial}. 
In \cite[Section~2]{Abe-Tagami}, Abe and the author proved that $6_{3}\neq A(6_{3})$. 
Hence, by Theorem~\ref{thm:trivial}, we have $A^{-1}(6_{3})=6_{3}$. 
We can prove this fact directly (see Figure~\ref{figure:ex2}). 
\begin{figure}[h]
\centering
\includegraphics[scale=0.65]{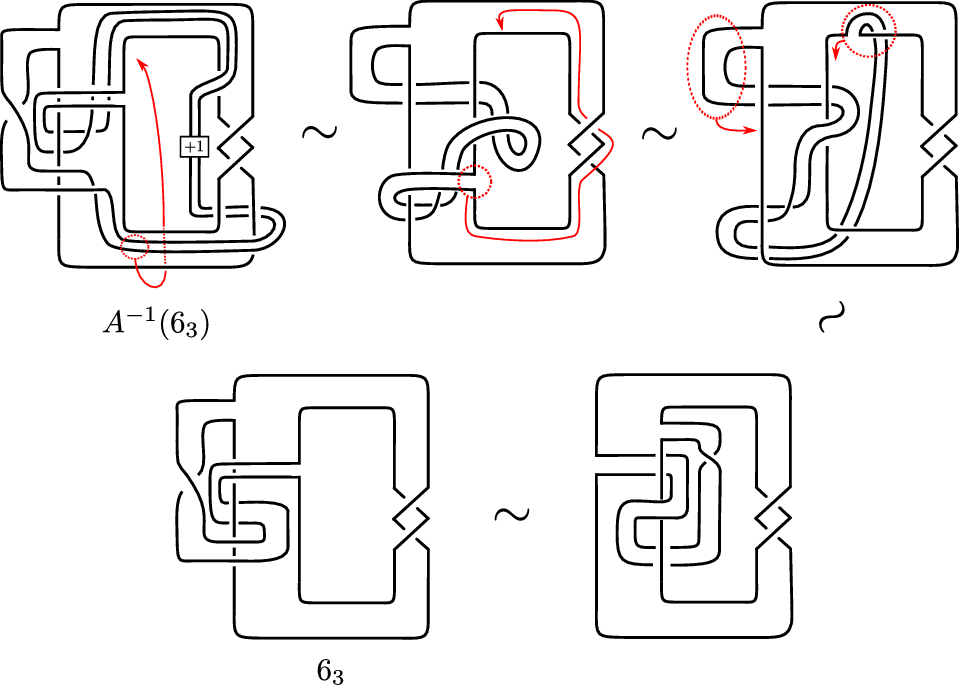}
\caption{
(color online) Direct proof for $A^{-1}(6_{3})=6_{3}$. Note that $6_3$ is invertible. 
The bottom left picture is obtained from the bottom right picture by flipping the annulus $A$. 
}\label{figure:ex2}
\end{figure}
\end{ex}
%
%
%
\noindent
\textbf{Acknowledgements.}
This work was supported by JSPS KAKENHI Grant numbers JP18K13416 and JP22K13923. 

\bibliographystyle{amsplain}
\bibliography{mrabbrev,tagami}
\end{document}